\documentclass[]{amsart}

\def\eqalignno#1{\displ@y \tabskip\@centering
  \halign to\displaywidth{\hfil$\@lign\displaystyle{##}$\tabskip\z@skip
    &$\@lign\displaystyle{{}##}$\hfil\tabskip\@centering
    &\llap{$\@lign##$}\tabskip\z@skip\crcr
    #1\crcr}}
\def\leqalignno#1{\displ@y \tabskip\@centering
  \halign to\displaywidth{\hfil$\@lign\displaystyle{##}$\tabskip\z@skip
    &$\@lign\displaystyle{{}##}$\hfil\tabskip\@centering
    &\kern-\displaywidth\rlap{$\@lign##$}\tabskip\displaywidth\crcr
    #1\crcr}}

\newenvironment{thmenum}{

\begin{enumerate}}
{\end{enumerate}}

\usepackage[dvips]{graphicx}
\usepackage{epsfig}
\usepackage{latexsym}
\usepackage{amsfonts}
\usepackage{amsthm}
\usepackage{amsmath}

\renewenvironment{proof}[1][]{\vskip-\lastskip\par\vskip6pt plus2pt
minus0pt\par%
\noindent\textit{Proof.}\enspace\ignorespaces}{\hfill$\Box$\par\vskip6pt
plus2pt minus0pt}

\numberwithin{equation}{section}

\newtheorem{theorem}[equation]{Theorem}

\newtheorem{lemma}[equation]{Lemma}

\newtheorem{proposition}[equation]{Proposition}

\newtheorem{corollary}[equation]{Corollary}

\theoremstyle{definition}
\newtheorem{definition}[equation]{Definition}

\newtheorem{remark}[equation]{Remark}

\newtheorem{notation}[equation]{Notation}

\newtheorem{conventions}[equation]{Conventions}

\newtheorem{convention}[equation]{Convention}

\begin{document}
\author{Petra Hitzelberger, Linus Kramer and Richard M. Weiss}
\thanks{{\it Address of the first two authors:}\
Fachbereich Mathematik und Informatik, Universit\"at M\"unster,
Einsteinstrasse~62, 48149 M\"unster, Germany} 
\thanks{{\it Address of the third author:}\
Department of Mathematics, Tufts University, 
503~Boston Avenue, Medford, MA 02155, USA}
\title{Non-discrete Euclidean Buildings for the Ree and Suzuki groups}
%\date{\today}
\begin{abstract}
We call a non-discrete Euclidean building a {\it Bruhat-Tits space} if
its automorphism group contains a subgroup that induces
the subgroup generated by all the root groups of a root datum of the
building at infinity. This is the class of non-discrete
Euclidean buildings introduced and studied by Bruhat and Tits
in \cite{bruh-tits}. We give the complete
classification of Bruhat-Tits spaces whose building at infinity is
the fixed point set of a polarity of 
an ambient building of type ${\sf B}_2$, ${\sf F}_4$ or ${\sf G}_2$
associated with a Ree or Suzuki group endowed with the usual root datum.
(In the ${\sf B}_2$ and ${\sf G}_2$ cases, this fixed point set is a building
of rank one; in the ${\sf F}_4$ case, it is a generalized
octagon whose Weyl group is not crystallographic.)
We also show that each of these Bruhat-Tits spaces has a 
natural embedding in the unique Bruhat-Tits space
whose building at infinity is the corresponding ambient building.
\end{abstract}

\maketitle

\pagestyle{myheadings}
\markboth{\uppercase{Petra Hitzelberger, Linus Kramer and Richard Weiss}}
{\uppercase{Non-discrete Euclidean Buildings}}

\section{Introduction}
Suppose that $X$ is an irreducible affine building.
Typically (and for certain if the dimension of $X$ is at least three), the automorphism group of 
the building at infinity of $X$ contains subgroups constituting
a {\it root datum} defined over a field $K$. In this case
the affine building $X$ is uniquely determined by
a valuation of this root datum which is, 
in turn, uniquely determined by a discrete valuation of $K$. 

The notion of a valuation of the root datum of a spherical building makes
perfectly good sense if we drop the requirement that the 
values lie in a discrete subgroup of ${\mathbb R}$.
In the non-discrete case, there is no longer a corresponding
affine building $X$. There does exist, 
however, an analogous structure called a {\it non-discrete Euclidean building}.

Non-discrete Euclidean buildings were
first introduced and studied by Bruhat and Tits in \cite{bruh-tits}.
These structures were first axiomatized and, in dimension greater than two,
classified by Tits in \cite{tits-como}.
(Other fundamental references about non-discrete Euclidean 
buildings are \cite{kleiner} and \cite{parreau}; see also \cite{berenstein} and \cite{rousseau}.)

Non-discrete Euclidean buildings are sometimes called 
{\it affine ${\mathbb R}$-buildings} (since
in the simplest case they are ${\mathbb R}$-trees) or {\it apartment systems} (since
they have apartments but are not really buildings) 
or, as in \cite{kleiner}, simply {\it Euclidean buildings}, although
this term is more commonly synonymous with ``affine building.'' 

In \ref{abc21} below, we propose the term {\it Bruhat-Tits space} for the class of
non-discrete Euclidean buildings that were introduced and studied by Bruhat and Tits in 
\cite{bruh-tits}. (The term {\it Bruhat-Tits space} was used in 
\cite{lang} to describe complete metric spaces that satisfy a certain 
``semi-parallelogram rule'' introduced by S. Lang.
This is a more general class of metric spaces which
includes not only all non-discrete Euclidean buildings but also, for example, all
simply connected Riemannian manifolds of non-positive curvature.
We mention too that Bruhat-Tits spaces in our sense are not necessarily
complete as metric spaces; see Section~7.5 in \cite{bruh-tits}.)

Suppose now that $\Delta$ is the spherical building (of rank one or two) associated with
a Ree or Suzuki group. Thus $\Delta$ is the fixed point set of a ``polarity''
of a building of type ${\sf B}_2$, ${\sf F}_4$ or ${\sf G}_2$ which is 
defined over a field extension $K/F$, where $p:={\rm char}(K)$ equals 2 in the first two
cases and 3 in the third case,
and the polarity is defined in terms of a Tits endomorphism $\theta$ of $K$ (as defined in
\ref{abc33}) 
whose image is $F$. On page~173 of \cite{tits-como}, Tits remarks that an arbitrary 
valuation $\nu$ of $K$ (where ``arbitrary'' means ``arbitrary non-trivial 
real-valued non-archimedean'')
extends to a valuation of the root datum of $\Delta$---and 
thus there exists a corresponding Bruhat-Tits space whose building
at infinity is $\Delta$---if and only if the valuation $\nu$ is $\theta$-invariant.
It is the goal of this paper to make this statement more precise and
to fill in all the details justifying it.

Here {\it $\theta$-invariant} means that $\nu\circ\theta$ is {\it equivalent}
to $\nu$, i.e.~that $\nu(x)\ge0$ for $x\in K^*$ if and only if $\nu(x^\theta)\ge0$.
This is the same as saying that $\nu(x^\theta)=\gamma\cdot\nu(x)$ for some
positive real number $\gamma$ and all $x\in K^*$. Since
$\theta$ is a Tits endomorphism of $K$, 
it follows that $\nu$ is $\theta$-invariant if and only if 
$$\nu(x^\theta)=\sqrt{p}\cdot\nu(x)$$ 
for all $x\in K^*$. Thus, in particular, $\nu$ cannot be $\theta$-invariant
if it is discrete (since the ratio of two values of a discrete valuation
is always rational).

\section{Overview}
Continuing to put precision aside for a moment, we can summarize
the main results of this paper roughly as follows; see also \ref{abc99}.

\begin{theorem}\label{abc300}
Let $G$ be a Ree or Suzuki group. Then the following hold:
\begin{thmenum}
\item There exists
a Moufang building $\dot{\Delta}$ of type ${\sf B}_2$, ${\sf F}_4$ or ${\sf G}_2$
having a polarity $\rho$ defined over a pair $(K,\theta)$ as described in 
Section~1 such that $G$ is the group induced by the centralizer
$C_{\dot{G}}(\rho)$ on the set $\dot{\Delta}^\rho$
of fixed points of $\rho$, where $\dot{G}$ is 
the subgroup of ${\rm Aut}(\dot{\Delta})$ generated by the root groups of $\dot{\Delta}$.
\item The set $\Delta:=\dot{\Delta}^\rho$ has the structure
of a Moufang building of rank one in cases ${\sf B}_2$ and 
${\sf G}_2$, of rank two whose Weyl group is dihedral of order 16 in 
case ${\sf F}_4$. 
\item For each valuation $\nu$ of the field $K$,
there exists a unique non-discrete Euclidean building $(\dot{X},\dot{\mathcal A})$
determined by $\nu$
whose building at infinity is $\dot{\Delta}$ and whose automorphism group
induces $\dot{G}$ on $\dot{\Delta}$.
\item Let $\nu$ and $(\dot{X},\dot{\mathcal A})$ be as in (iii). Then 
there exists an automorphism $\dot{\rho}$ of $(\dot{X},\dot{\mathcal A})$ 
inducing the polarity $\rho$ on $\dot{\Delta}$ if and only if $\nu$ is $\theta$-invariant.
Furthermore, $\dot{\rho}$, if it exists, is unique.
\item If $\nu$, $(\dot{X},\dot{\mathcal A})$ and $\dot{\rho}$ are as in (iv),
then there is a canonical non-discrete Euclidean building $(X,{\mathcal A})$ contained 
in the fixed point set of $\dot{\rho}$ in $(\dot{X},\dot{\mathcal A})$ whose
building at infinity is $\Delta$ and whose automorphism group contains
a subgroup inducing $G$ on $\Delta$.
\item Every non-discrete Euclidean building $(X,{\mathcal A})$ whose building at 
infinity is $\Delta$ and whose automorphism group contains a subgroup
inducing $G$ on $\Delta$ arises from a $\theta$-invariant valuation 
$\nu$ of $K$ as described in (v).
\end{thmenum}
\end{theorem}

\begin{proof}
Assertion (i) is essentially the definition of a Ree or Suzuki group.
The building $\dot{\Delta}$ and its polarity $\rho$ are described 
in \ref{abc6}. Assertion (ii) is proved in \ref{abc29}.iii and \ref{abc49}.
Assertions (iii) and (iv) follow from  \ref{abc31} and \ref{abc80}.
Assertion (v) is proved in \ref{abc93} and assertion (vi)
is a consequence of \ref{abc90}.
\end{proof}

This paper is organized as follows:
In Sections~3--4 we review the basic results about root data, valuations of root
data and non-discrete Euclidean buildings we require. In Sections~5 we introduce
the spherical buildings of type ${\sf B}_2$, ${\sf F}_4$ and ${\sf G}_2$
having polarities that give rise to the Ree and Suzuki groups and in Section~6 
we assemble the basic properties of the subbuildings fixed by these polarities
we require.
Our main results---\ref{abc90}, \ref{abc91} and \ref{abc93}---are then proved 
in Sections~7 and~8. 

Throughout this paper we will use $\Delta$
and similar letters to denote spherical buildings and $(X,{\mathcal A})$
and similar letters to denote Euclidean buildings.

\section{Root data and valuations}
We now start paying attention to the details. 
In this section we review the notions of a root datum of a spherical building
and a valuation of a root datum.

\begin{notation}\label{abc0}
Let $(W,S)$ be an irreducible spherical Coxeter system. Then either $W$
can be identified with the Weyl group of an irreducible root system $\Phi$
so that $S$ consists of the reflections determined by the elements in a basis,
or $|S|=2$ and $W$ is a dihedral group of
order $2n$ for $n=5$ or $n>6$. In the latter case we let 
$\Phi$ consist of $2n$ vectors evenly distributed around
the unit circle in a 2-dimensional Euclidean space and think of $S$
as the two reflections determined by two of these vectors making an angle
of $(n-1)180/n$ degrees. (Later we will refer to this set $\Phi$
as ${\sf I}_2(n)$.)
In both cases, we denote by ${\mathbb A}$ the ambient Euclidean space of $\Phi$
and by ${\rm Aut}(\Phi)$ the group of isometries of ${\mathbb A}$
mapping $\Phi$ to itself. Note that $(W,S)$ is uniquely determined by 
$\Phi$. When we sometimes call $W$ the {\it Weyl group} of 
$\Phi$ (as is usual), we really have the pair $(W,S)$ in mind.
\end{notation}

\begin{notation}\label{abc16}
Let $(W,S)$ and $\Phi$ be as in \ref{abc0}.
For all $\alpha,\beta\in\Phi$ such that $\alpha\ne\pm\beta$, the
{\it interval} $(\alpha,\beta)$ is the $s$-tuple
$(\gamma_1,\gamma_2,\ldots,\gamma_s)$ of all elements $\gamma_i\in\Phi$ such
that for some positive real numbers $p_i$ and $q_i$ (which depend
on $\alpha$ and $\beta$),
\begin{equation}\label{abc16b}
\gamma_i/|\gamma_i|=p_i\alpha/|\alpha|+q_i\beta/|\beta|,
\end{equation}
where $\angle(\alpha,\gamma_i)<\angle(\alpha,\gamma_j)$ if and only if $i<j$.
Note that $s$ depends on the pair $\alpha$ and $\beta$ and that for some pairs, $s=0$.
\end{notation}

To define the interval $(\alpha,\beta)$, we could
have omitted the denominators in \ref{abc16b}. We included the denominators
because the coefficients $p_i$ and $q_i$ in this equation (as it is written)
are needed in the statement of condition (V2) in \ref{abc14}.

\begin{notation}\label{abc60}
Let $(W,S)$, $\Phi$ and ${\mathbb A}$
be as in \ref{abc0}. To each $\alpha\in\Phi$ we associate the reflection
$s_\alpha$ given by 
$$s_\alpha(v)=v-2(v\cdot\alpha)\alpha$$
for each $v\in{\mathbb A}$. A {\it wall} of $\Phi$ is the fixed point set
of one of these reflection. A {\it Weyl chamber} of $\Phi$ is a 
connected component of ${\mathbb A}$ with all the walls removed.
We call the closure of a Weyl chamber a {\it sector} of $\Phi$.
The sectors are polyhedral cones; a {\it face} of $\Phi$ is a
face of one of these cones. The set of all faces of $\Phi$ forms
a simplicial complex called the {\it Coxeter complex} of $(W,S)$ (or $\Phi$).
We will denote this simplicial complex by 
$\Sigma(\Phi)$. To each
element $\alpha$ of $\Phi$ we associate the half-space
$$H_\alpha:=\{v\in{\mathbb A}\mid v\cdot\alpha\ge0\}.$$
A {\it root} of the Coxeter complex is the set of faces contained
in one of these half-spaces. The map $\alpha\mapsto H_\alpha$ thus gives
rise to a canonical bijection 
from $\Phi$ to the set of roots of $\Sigma(\Phi)$.
\end{notation}

\begin{conventions}\label{abc70}
Let $\Delta$ be an irreducible spherical building. 
Then every apartment of $\Delta$ is isomorphic to the Coxeter complex $\Sigma(\Phi)$
for some $\Phi$ as in \ref{abc0}. The corresponding
Coxeter system $(W,S)$ is usually called the {\it type} of $\Delta$.
We prefer, instead, to say that $\Delta$ is {\it of type $\Phi$}.
(Thus a building of type ${\sf B}_\ell$ is the same thing as building
of type ${\sf C}_\ell$.)
Let $\Sigma$ be an apartment of $\Delta$. A {\it root} of $\Sigma$ is the image
of a root of $\Sigma(\Phi)$ under an isomorphism from $\Sigma(\Phi)$ to $\Sigma$.
Thus for each such isomorphism we
have a canonical bijection from $\Phi$ to the set of roots of $\Sigma$.
We will usually assume that an isomorphism from $\Sigma(\Phi)$ to $\Sigma$ is fixed and
identify $\Phi$ with the set of roots of $\Sigma$ via this bijection.
\end{conventions}

\begin{definition}\label{abc2}
Let $\Delta$ be an irreducible spherical building of rank at least two.
For each root $\alpha$ of $\Delta$ (i.e.~of some apartment of $\Delta$), let 
$U_\alpha$ be the intersection of the stabilizers in ${\rm Aut}(\Delta)$ of 
all the chambers that are contained in some
panel of $\Delta$ which contains two chambers in $\alpha$. It follows from 
Corollary~3.14 in \cite{me-sph} that
$U_\alpha$ acts trivially on the set of chambers in $\alpha$ itself.
(The subgroups of the form $U_\alpha$ are called the {\it root groups} of $\Delta$.)
The building $\Delta$ is {\it Moufang} 
(equivalently, ``$\Delta$ satisfies the {\it Moufang condition}'')
if the following hold:
\begin{thmenum}
\item $\Delta$ is thick (i.e.~every panel contains at least three chambers).
\item For each root $\alpha$ of $\Delta$,
the root group $U_\alpha$ acts transitively on the set of apartments containing $\alpha$.
\end{thmenum}
\end{definition}

\noindent
Tits showed that thick
irreducible spherical buildings of rank at least three, as well as all
the irreducible residues of rank two of such a building, {\it always} satisfy
the Moufang condition; see, for example, Theorems~11.6 and~11.8 in \cite{me-sph}
for a proof.
Irreducible spherical buildings of rank at least two satisfying the 
Moufang condition were classified in \cite{tits-sph} and \cite{TW}.
See \cite[30.14]{weiss} for a summary of the results. 

\begin{proposition}\label{abc8}
Let $\Delta$ be an irreducible spherical building of rank $\ell\ge2$ satisfying
the Moufang property. Then for each root $\alpha$ of $\Delta$, 
the root group $U_\alpha$ acts sharply transitively
on the set of apartments containing $\alpha$.
\end{proposition}

\begin{proof} 
This is proved, for example, in Theorem~9.3 and Proposition~11.4 of \cite{me-sph}.
\end{proof}

Now suppose that $\Delta$ is a building of rank one. In other words,
$\Delta$ is simply a set (whose elements are the chambers of $\Delta$) 
and ${\rm Aut}(\Delta)$ is 
the full symmetric group on $\Delta$. The apartments of $\Delta$ 
are the two-element subsets of $\Delta$ and thus the roots of $\Delta$ are just the one-element
subsets of $\Delta$. Normally we will use letters like $x$ and $y$ to name elements of $\Delta$;
when we want to emphasize that an element of $\Delta$ is being considered
as a root, however, we will give it a name like $\alpha$ or $\beta$.

\begin{definition}\label{abc10}
Let $\Delta$ be a building of rank one, i.e.~of type $\Phi:={\sf A}_1$,
let $\Sigma$ be an apartment of $\Delta$ and let the two elements of 
$\Sigma$ be identified with the two elements of $\Phi$.
A {\it Moufang structure} on $\Delta$ is a collection 
$$(U_\alpha)_{\alpha\in\Phi}$$ 
of non-trivial subgroups of ${\rm Aut}(\Delta)$ such that the following hold:
\begin{thmenum}
\item For each $\alpha\in\Phi$, the subgroup $U_\alpha$ fixes $\alpha$ and
acts sharply transitively on $\Delta\backslash\{\alpha\}$; and
\item For each $\alpha\in\Phi$, the subgroup $U_\alpha$ is normalized by
the stabilizer of $\alpha$ in 
the group $G:=\langle U_\alpha,U_{-\alpha}\rangle$.
\end{thmenum}
The groups $U_\alpha^g$ for all $\alpha\in\Phi$ and all $g\in G$ are called
{\it root groups}. A Moufang structure on $\Delta$ is independent of the choice
of $\Sigma$ and its identification with $\Phi$ up to conjugation in the 
group $G$.
Since the root groups are required to be non-trivial, $\Delta$ can have a Moufang structure
only if it is thick, i.e.~if $|\Delta|\ge3$. When we say that $\Delta$ is Moufang, 
we mean that we have a particular Moufang structure on $\Delta$ in mind
(whose root groups we will always call 
$U_\alpha$, $U_\beta$, etc.). Note that with this convention,
\ref{abc8} holds also when $\ell=1$ (with ``satisfying the Moufang property'' interpreted
as meaning ``having a Moufang structure with root groups $U_\alpha$).
\end{definition}

\begin{remark}\label{abc19}
Let $\Delta$ be an irreducible spherical building of rank $\ell\ge2$ which satisfies
the Moufang condition and let $P$ be a panel of $\Delta$ viewed as a set of chambers.
For each chamber $x$ in $P$, let $\alpha$ be an arbitrary root of $\Delta$ containing
$x$ but no other chamber in $P$. By \ref{abc8},
$U_\alpha$ acts faithfully on $P$. Furthermore, the permutation group induced
by the root group $U_\alpha$ on $P$ is independent of the choice of 
$\alpha$. (This follows from Proposition~11.11 in \cite{me-sph}.) 
Hence every rank one residue of $\Delta$ inherits a canonical Moufang structure from $\Delta$
whose root groups are isomorphic to root groups of $\Delta$.
\end{remark}

\begin{proposition}\label{abc8a}
Let $\Delta$ be an irreducible spherical building of type $\Phi$
satisfying the Moufang condition as defined in \ref{abc2} and \ref{abc10} and
let $\Sigma$ be an apartment of $\Delta$ (to which we apply \ref{abc70}).
Then for each $\alpha\in\Phi$ (i.e.~to each root $\alpha$ of $\Sigma$), the following hold:
\begin{thmenum}
\item There exist maps $\lambda$ and $\kappa$ from $U_\alpha^*$ to $U_{-\alpha}^*$
such that for each $u\in U_\alpha^*$, the product
$$m_\Sigma(u):=\kappa(u)u\lambda(u)$$
maps $\Sigma$ to itself and induces the unique reflection $s_\alpha$ defined
in \ref{abc60} on $\Phi$ (which interchanges the roots $\alpha$ and $-\alpha$).
\item $m_\Sigma(u)^{-1}=m_\Sigma(u^{-1})$ for each $u\in U_\alpha$.
\item $m_\Sigma(\kappa(u))=m_\Sigma(\lambda(u))=m_\Sigma(u)$ for all $u\in U_\alpha$.
\end{thmenum}
\end{proposition}

\begin{proof}
The first assertion is a consequence of \ref{abc8} and the other two follow
from the first; see, for example, in \cite[6.1--6.3]{TW}.
\end{proof}

\begin{proposition}\label{abc13}
Let $\Delta$ be an irreducible spherical building 
satisfying the Moufang condition as defined in \ref{abc2} and \ref{abc10}
and let $G^\dagger$ be 
the subgroup of ${\rm Aut}(\Delta)$ generated by all the root groups of $\Delta$. Then $G^\dagger$
acts transitively on the set of all
pairs $(\Sigma,C)$, where $\Sigma$ is an apartment of $\Delta$ and $C$ is a chamber of 
$\Sigma$. 
\end{proposition}

\begin{proof} 
This is proved, for example, in Proposition~11.12 of \cite{me-sph}.
\end{proof}

\noindent
Note in the rank one case, \ref{abc13}
just says that $G^\dagger$ is a 2-transitive permutation group on $\Delta$.

\begin{definition}\label{abc12}
Let $\Delta$ be an irreducible spherical building of type 
$\Phi$ satisfying the Moufang condition and
let $\Sigma$ be an apartment of $\Delta$ (to which we apply \ref{abc70}).
For each  $\alpha\in\Phi$, let $U_\alpha$ be the
corresponding root group. The {\it root datum} of $\Delta$ (based at $\Sigma$)
is the pair $(\Sigma,(U_\alpha)_{\alpha\in\Phi})$. 
\end{definition}

Let $\Delta$ be as in \ref{abc12}. By \ref{abc13}, the root datum of $\Delta$ is, up to 
conjugation in the group $G^\dagger$, independent of the choice of the apartment $\Sigma$. 
Note that a root datum and a Moufang structure on $\Delta$
(as defined in \ref{abc10}) are essentially
the same thing when $\Delta$ has rank one.
By \cite[40.17]{TW},
$\Delta$ is uniquely determined by its root datum when the rank of $\Delta$ is at least two.

\begin{theorem}\label{abc17}
Let $\Delta$ be an irreducible spherical building of type $\Phi$ 
satisfying the Moufang condition,
let the notion of the interval from one element of $\Phi$ to another be as in \ref{abc16}
and let $\Sigma$ be an apartment of $\Delta$ (to which we apply \ref{abc70}).
Then for all ordered pairs $\alpha,\beta$ of elements of $\Phi$ such that 
$\beta\ne\pm\alpha$,
$$[U_\alpha,U_\beta]\subset U_{\gamma_1}U_{\gamma_2}\cdots U_{\gamma_s},$$
where $(\gamma_1,\gamma_2,\ldots,\gamma_s)$ is the interval $(\alpha,\beta)$,
if the interval $(\alpha,\beta)$ is not empty and $[U_\alpha,U_\beta]=1$ if
it is.
\end{theorem}

\begin{proof}
This is proved in \cite[6.12(ii)]{ronan}.
\end{proof}

\begin{definition}\label{abc14}
Let $\Delta$ be an irreducible spherical building of type $\Phi$ satisfying the Moufang condition,
let $\Sigma$ be an apartment of $\Delta$ (to which we apply \ref{abc70})
and let
$$(\Sigma,(U_\alpha)_{\alpha\in\Phi})$$
be the root datum of $\Delta$ based at $\Sigma$ as defined in \ref{abc12}.
A {\it valuation} of
this root datum is a collection $\varphi:=(\varphi_\alpha)_{\alpha\in\Phi}$ of 
non-constant maps $\varphi_\alpha$ from $U_\alpha^*$ to ${\mathbb R}$
such that the following hold:
\begin{thmenum}
\item[{\bf V1:}] For each $\alpha\in\Phi$ and each $k\in{\mathbb R}$, the set
$$U_{\alpha,k}:=\{u\in U_\alpha\mid\varphi_\alpha(u)\ge k\}$$
is a subgroup of $U_\alpha$, where we assign $\varphi_\alpha(1)$ the value $\infty$
(so that $1\in U_{\alpha,k}$ for all $k$).
\item[{\bf V2:}] For all $\alpha,\beta\in\Phi$ such that $\alpha\ne\pm\beta$ and 
for all $k,l\in{\mathbb R}$,
$$[U_{\alpha,k},U_{\beta,l}]=U_{\gamma_1,p_1k+q_1l}U_{\gamma_2,p_2k+q_2l}\cdots
U_{\gamma_s,p_sk+q_sl},$$
where $\gamma_i$, $p_i$, $q_i$ and $s$ are as in \ref{abc16b}.
\item[{\bf V3:}] For all $\alpha,\beta\in\Phi$, all $u\in U_\alpha^*$ and all 
$g\in U_\beta$, the quantity
$$t:=\varphi_{s_\alpha(\beta)}(g^{m_\Sigma(u)})-\varphi_\beta(g)$$
is independent $g$ and if $\alpha=\beta$, then $t=-2\varphi_\alpha(u)$.
Here $s_\alpha$ is as in \ref{abc60} and $m_\Sigma(u)$ is as in \ref{abc8a}.i. 
\end{thmenum}
\end{definition}

\noindent
Note that the condition (V2) is vacuous when the rank of $\Delta$ is one.

\begin{definition}\label{abc30}
Let $\Delta$, $\Phi$ and $\Sigma$ be as in \ref{abc14} and 
suppose that $\varphi$ and $\psi$ are two valuations of the root datum
of $\Delta$ based at $\Sigma$. Then $\varphi$ and $\psi$ are {\it equipollent}
if for some $x$ in the ambient Euclidean space ${\mathbb A}$ of the root system $\Phi$,
$$\varphi_\alpha(u)=\psi_\alpha(u)+\alpha\cdot x$$
for all $\alpha\in\Phi$ and all $u\in U_\alpha^*$ (in which case we write
$\varphi=\psi+x$).
\end{definition}

\begin{proposition}\label{abc84}
Let $\Delta$, $\Phi$ and $\Sigma$ be as in \ref{abc14} and 
suppose that $\varphi$ and $\psi$ are two valuations of the root datum
of $\Delta$ based at $\Sigma$ such that $\varphi_\alpha=\psi_\alpha$ 
for some $\alpha\in\Phi$. Then $\psi$ and $\varphi$ are equipollent (as defined
in \ref{abc30}).
\end{proposition}

\begin{proof}
This holds by Proposition~6 in \cite{tits-como}. For more details, see
Theorem~3.41 in \cite{weiss}.
\end{proof}

\begin{definition}\label{abc32}
Let $\Delta$, $\Phi$ and $\Sigma$ be as in \ref{abc14}, 
let $\varphi=(\varphi_\alpha)_{\alpha\in\Phi}$ be a valuation 
of the root datum of $\Delta$ based at $\Sigma$ and let $\rho$ be an automorphism
of $\Delta$ mapping $\Sigma$ to itself. We will say that $\varphi$ is
{\it $\rho$-invariant} if 
$$\varphi_{\alpha}(u)=\varphi_{\alpha^\rho}(u^\rho)$$ 
for all $\alpha\in\Phi$ and all $u\in U_\alpha^*$.
\end{definition}

\begin{proposition}\label{abc40}
Let $\Delta$, $\Phi$ and $\Sigma$ be as in \ref{abc14},
let $\varphi$ be a valuation of the 
root datum of $\Delta$ based at $\Sigma$, let $w$ be an element of the root group 
$U_\alpha^*$ such that $\varphi_\alpha(w)=0$ 
and let $m_0=m_\Sigma(w)$. Then for each $\alpha\in\Phi$,
$$\varphi_\alpha(g^{m_0m_\Sigma(u)})-\varphi_\alpha(g)=2\varphi_\alpha(u)$$
for all $g,u\in U_\alpha^*$.
\end{proposition}

\begin{proof}
Let $\beta$ be the root opposite $\alpha$ in $\Sigma$.
Let $g,u\in U_\alpha^*$ and let $v=\kappa(u)\in U_\beta^*$, 
where $\kappa$ is as in \ref{abc8a}.i.
Then $m_\Sigma(u)=m_\Sigma(v)$ by Proposition~11.24 in \cite{me-sph}, 
$\varphi_\beta(v)=-\varphi_\alpha(u)$ by Proposition~3.25 in \cite{weiss} 
(or \cite[10.10]{ronan}) and
$$\varphi_\alpha(y^{m_\Sigma(v)})=\varphi_\beta(y)-2\varphi_\beta(v)$$ 
for all $y\in U_\beta$ by condition (V3) (with both roots equal to $\beta$).
Thus 
\begin{align*}
\varphi_\alpha(g^{m_0m_\Sigma(u)})&=\varphi_\alpha(g^{m_0m_\Sigma(v)})\cr
&=\varphi_\beta(g^{m_0})-2\varphi_\beta(v)\cr
&=\varphi_\beta(g^{m_0})+2\varphi_\alpha(u)
\end{align*}
By another application of condition (V3) (this time with both roots equal to $\alpha$)
and the choice of $w$, we have
$$\varphi_\beta(g^{m_0})=\varphi_\alpha(g).$$
\end{proof}

\section{Non-discrete Euclidean Buildings}
In this section we assemble a few basic facts about non-discrete Euclidean
buildings. We start with the definition.

\begin{notation}\label{abc61}
Let $W$, ${\mathbb A}$ and $\Phi$ be as in \ref{abc0}.
Let ${\mathbb W}$ denote the group generated by 
$W$ and the group $T$ consisting of all translations of ${\mathbb A}$. 
Thus ${\mathbb W}$ is a group of isometries of ${\mathbb A}$. Moreover,
${\mathbb W}=TW$ and $T$ is a normal subgroup of ${\mathbb W}$. 
\end{notation}

\begin{definition}\label{abc1}
Let $\Phi$ and $({\mathbb A},{\mathbb W})$ be as in \ref{abc61} and
let $H_\alpha$ for all $\alpha\in\Phi$ be as in \ref{abc60}.
Let $X$ be a set and 
let ${\mathcal A}$ be a family of injections
of ${\mathbb A}$ into $X$. The elements of ${\mathcal A}$ 
will be called {\it charts} and the images of charts
will be called {\it apartments}. Sets of the form $f(H_\alpha)$
for some chart $f$ and some $\alpha\in\Phi$
will be called {\it roots} of $(X,{\mathcal A})$
and sets of the form $f(S)$ for some chart $f$
and some sector (respectively, face) of $\Phi$ (as defined in \ref{abc60})
will be called {\it sectors} (respectively, {\it faces}) of 
$(X,{\mathcal A})$. The pair $(X,{\mathcal A})$ is
a {\it non-discrete Euclidean building of type $\Phi$}
if the following six axioms hold:
\begin{thmenum}
\item[{\bf A1}:] If $f\in{\mathcal A}$ and $w\in{\mathbb W}$, then $f\circ w\in{\mathcal A}$.
\item[{\bf A2}:] If $f,f'\in{\mathcal A}$, then the set 
$$M:=\{v\in{\mathbb A}\mid f(v)\in f'({\mathbb A})\}$$
is closed and convex and there exists $w\in{\mathbb W}$ such that the maps $f$ and $f'\circ w$
coincide on $M$.
\item[{\bf A3}:] Every two points of $X$ are contained in a
common apartment.
\item[{\bf A4}:] If $S$ and $S'$ are sectors of $(X,{\mathcal A})$, then
there exists an apartment containing sectors $S_1$ and $S_1'$ such 
that $S_1\subset S$ and $S_1'\subset S'$.
\item[{\bf A5}:] Three apartments which intersect pairwise in roots have a 
non-empty intersection.
\item[{\bf A6}:] There is a metric $d$ on $X$ 
such that for all $v,z\in{\mathbb A}$ and all $f\in{\mathcal A}$,
$d(f(v),f(z))$ equals the Euclidean distance between $v$ and $z$.
\end{thmenum}
If $(X,{\mathcal A})$ is a non-discrete Euclidean building of type $\Phi$,
we call the pair $({\mathbb A},{\mathbb W})$, which is uniquely determined
by $\Phi$, its {\it model} and we define the {\it dimension} of $(X,{\mathcal A})$
to be the dimension of ${\mathbb A}$. By (A3), the metric $d$ in (A6) is 
unique.
\end{definition}

Various equivalent definitions of a non-discrete Euclidean
building (and a proof of their equivalence) can be found
in Theorem~1.21 of \cite{parreau}. See also 
Proposition~2.21 of \cite{parreau}. 

\begin{definition}\label{abc3}
Let $(X,{\mathcal A})$ and $(X',{\mathcal A}')$ be two non-discrete Euclidean
buildings having the same type $\Phi$. An {\it isomorphism} $\psi$ from 
$(X,{\mathcal A})$ to $(X',{\mathcal A}')$ is a bijection
from $X$ to $X'$ such that 
$${\mathcal A}'=\{\psi\circ f\mid f\in {\mathcal A}\}.$$
We denote by ${\rm Aut}(X,{\mathcal A})$ the group of all
isomorphisms from $(X,{\mathcal A})$ to itself.
\end{definition}

\begin{definition}\label{abc5}
Let $(X,{\mathcal A})$ be a non-discrete Euclidean building of type $\Phi$,
let $({\mathbb A},{\mathbb W})$ be the model of $(X,{\mathcal A})$,
let $\tau$ be an isometry of ${\mathbb A}$ normalizing the group ${\mathbb W}$ and let 
$${\mathcal A}_\tau=\{f\circ\tau\mid f\in {\mathcal A}\}.$$
Then $(X,{\mathcal A}_\tau)$ is a non-discrete Euclidean building of type $\Phi$ with the same
underlying metric structure as $(X,{\mathcal A})$.
Now suppose in addition that $\tau^2=1$. 
Then an isomorphism from $(X,{\mathcal A})$ to $(X,{\mathcal A}_\tau)$ 
is automatically an isomorphism from $(X,{\mathcal A}_\tau)$ to $(X,{\mathcal A})$. 
We call such an isomorphism a {\it $\tau$-automorphism} of $(X,{\mathcal A})$.
\end{definition}

\begin{definition}\label{abc67}
We say that two non-discrete Euclidean buildings
$(X,{\mathcal A})$ and $(X',{\mathcal A}')$ are {\it equivalent} (or one is a 
{\it dilation} of the other) if 
they have the same type $\Phi$ and therefore the same model $({\mathbb A},{\mathbb W})$,
$X=X'$ and 
$${\mathcal A}'=\{f\circ\delta\mid f\in{\mathcal A}\}$$
for some dilation $\delta$ of ${\mathbb A}$ (where {\it dilation} means
multiplication by a non-zero constant). 
Thus equivalent non-discrete Euclidean buildings
have the same underlying metric structure up to a constant positive factor.
\end{definition}

\begin{definition}\label{abc73}
Let $(X,{\mathcal A})$ be a non-discrete Euclidean 
building with model $({\mathbb A},{\mathbb W})$
and let $x,x'\in X$. By (A3), there exists a chart $f$ and points $x_1,x_1'\in{\mathbb A}$
such that $f(x_1)=x$ and $f(x_1')=x'$. The {\it interval} $[x,x']$ is the image
under $f$ of the interval $[x_1,x_1']$. By (A2), the interval
$[x,x']$ is independent of the chart $f$. 
By the CAT(0) property (proved, for example, in Proposition~2.10 of [5]),
the interval $[x,x']$ is, in fact, the unique
geodesic connecting $x$ to $x'$.
\end{definition}

\begin{notation}\label{abc57}
Let $(X,{\mathcal A})$ be a non-discrete Euclidean building of type $\Phi$. 
Two faces $F$ and $F'$ of $(X,{\mathcal A})$ (as defined in \ref{abc1}) 
are called {\it parallel} if they are at finite Hausdorff distance, i.e.~if both
$$\sup_{x'\in F'}d(x',F)$$
and
$$\sup_{x\in F}d(x,F')$$ 
are finite. By (A6), this is an equivalence relation on faces. 
For each face $F$ of $(X,{\mathcal A})$,
we denote by $F^\infty$ the corresponding parallel class and 
for each apartment $A$ of $(X,{\mathcal A})$, 
we denote by $A^\infty$ the set of parallel classes of faces
containing a face of $A$. If
$b$ and $b'$ are two parallel classes, we set $b\le b'$ whenever
$$\sup_{x'\in F'}d(x',F)<\infty.$$
for all $F\in b$ and all $F'\in b'$.
This notion makes the set of parallel classes of faces into a simplicial complex. 
We denote this simplicial complex by $(X,{\mathcal A})^\infty$.
By Proposition~1 in \cite{tits-como} (see also Property~1.7 in \cite{parreau}), 
$(X,{\mathcal A})^\infty$
is, in fact, a spherical building of type $\Phi$ (as defined in \ref{abc60})
and the map $A\mapsto A^\infty$
is a bijection from the set of apartments of $(X,{\mathcal A})$ to
the set of apartments of $(X,{\mathcal A})^\infty$.
The building $(X,{\mathcal A})^\infty$
is called the {\it building at infinity of $(X,{\mathcal A})$}. 
It is irreducible (since it is of type $\Phi$) and its rank
is the same as the dimension
of $(X,{\mathcal A})$.
\end{notation}

\begin{notation}\label{abc64}
Let $(X,{\mathcal A})$ be a non-discrete Euclidean building of type $\Phi$ with model
$({\mathbb A},{\mathbb W})$, let $A$ be
an apartment of $(X,{\mathcal A})$, let $x$ be a point of $A$ and 
let $\Sigma=A^\infty$. Let $\Sigma(\Phi)$ and $H_\alpha$ (for each $\alpha\in\Phi$) 
be as in \ref{abc60} and suppose that $\Phi$ is identified with the set of roots
of $\Sigma$ via an isomorphism from $\Sigma(\Phi)$ to $\Sigma$ as described in 
\ref{abc70}. Then there exists a unique chart $f$ such that the following hold:
\begin{thmenum}
\item $f({\mathbb A})=A$.
\item $f(0)=x$.
\item For each $\alpha\in\Phi$, a sector $S$ of $\Phi$ is contained in 
the half-space $H_\alpha$ if and only if the chamber $f(S)^\infty$ of $\Sigma$
is contained in the root $\alpha$ of $\Sigma$. 
\end{thmenum}
Let $f_{A,x}$ denote the chart $f$. 
\end{notation}

\begin{remark}\label{abc58}
Let $\Phi$ and $({\mathbb A},{\mathbb W})$ be as in \ref{abc61}.
Then $({\mathbb A},{\mathbb W})$ is a non-discrete Euclidean building of type
$\Phi$ with just one apartment. It thus has a building at infinity
whose faces are of the form $F^\infty$ for some face $F$ of $\Phi$.
\end{remark}

\begin{notation}\label{abc59}
Let $(X,{\mathcal A})$ be a non-discrete Euclidean building of type $\Phi$ 
and let $({\mathbb A},{\mathbb W})$ be its model.
Let $F$ be a face of $(X,{\mathcal A})$. Then $F$ is the image
of a face of $\Phi$ (as defined in \ref{abc60}) under some chart
$f$. The {\it vertex} of $F$ is the image of the origin of ${\mathbb A}$
under $f$; by (A2), this notion is independent of the choice
of $f$. Now let $x$ be a point of $X$.
We declare two points $y$ and $z$ of 
$X\backslash\{x\}$ to be {\it equivalent at $x$}
if 
$$[x,y]\cap B=[x,z]\cap B$$ 
for some open ball $B$ centered at $x$, where $[x,y]$ and $[x,z]$ are intervals
as defined in \ref{abc73}.
This is an equivalence relation on $X\backslash\{x\}$. For each 
$y\in X\backslash\{x\}$, let $g_x(y)$ denote its equivalence class.
We declare two faces $F$ and $F_1$ with vertex $x$ to be equivalent if $g_x(F)=g_x(F_1)$.
By (A2), this holds if and only if $F\cap B=F_1\cap B$ for
some open ball $B$ centered at $x$.
The equivalence class of a face $F$ is called the {\it germ of $F$} and
a {\it germ at $x$} is the germ of a face with vertex $x$. 
The set of germs at $x$ form a simplicial complex on the set $g_x(X\backslash\{x\})$.
As observed in Section 1.3 of \cite{parreau}, this simplicial
complex is a building of type $\Phi$ whose apartments are the sets
$g_x(A)$ for all apartments $A$ of $(X,{\mathcal A})$ containing $x$.
We call this building the 
{\it residue of $(\Delta,{\mathcal A})$ at $x$}
and denote it by $(\Delta,{\mathcal A})_x$. (The residue at $x$
is called the {\it building of directions} at $x$ in \cite{parreau}.)
\end{notation}

The residues of a non-discrete Euclidean building $(X,{\mathcal A})$ 
are not necessarily thick, and the residues at different points of $X$
are not necessarily isomorphic to each other. They are
all, however, of type $\Phi$.

\begin{notation}\label{abc63}
Let $(X,{\mathcal A})$ be a non-discrete Euclidean building, let $x\in X$,
let the residue $(X,{\mathcal A})_x$ of $(X,{\mathcal A})$ at $x$ be as in \ref{abc59} and
let $f$ be an element of ${\mathcal A}$ mapping the origin $0$ to $x$.
Then there exists a unique type-preserving isomorphism from the Coxeter complex
$\Sigma(\Phi)$ (as defined in \ref{abc60}) to an apartment of
the residue $(X,{\mathcal A})_x$ which maps each face $F$ 
of $\Phi$ to the germ at $x$ containing $f(F)$. We denote this isomorphism by 
$f_*$.
\end{notation}

For the rest of this section, we examine the special case that the 
building at infinity of a non-discrete Euclidean building is Moufang.

\begin{theorem}\label{abc20}
Let $(X,{\mathcal A})$ be a non-discrete Euclidean building of type $\Phi$
with model $({\mathbb A},{\mathbb W})$
such that the building at infinity $\Delta:=(X,{\mathcal A})^\infty$
satisfies the Moufang property as defined in \ref{abc2}, let
$\ell$ denote the dimension of $(X,{\mathcal A})$, let $A$ be an 
apartment of $(X,{\mathcal A})$, let $\Sigma=A^\infty$, let $x$ be a point of $A$,
let $\Sigma(\Phi)$ be as in \ref{abc60} and let 
$$H_{\alpha,k}=\{v\in{\mathbb A}\mid v\cdot\alpha\ge k\}$$
for all $\alpha\in\Phi$ and all $k\in{\mathbb R}$.
Let $\Phi$ be identified with the set of roots of $\Sigma$ via 
an isomorphism from $\Sigma(\Phi)$ to $\Sigma$ as described in
\ref{abc70} and let $f:=f_{A,x}$ be as in \ref{abc64}. 
If $\ell\ge2$, then the following hold:
\begin{thmenum}
\item For every root $\alpha\in\Phi$, there exists a canonical injection
from the root group $U_\alpha$ into ${\rm Aut}(X,{\mathcal A})$
such that for each $u\in U_\alpha$, its image under this injection 
induces $u$ on $\Delta$.
\item For every $\alpha\in\Phi$, there exists a map
$\varphi_\alpha$ from the root group $U_\alpha^*$ of $\Delta$ to ${\mathbb R}$ such that 
$${\rm Fix}_A(u)=A\cap A^u=f(H_{\alpha,\varphi_\alpha(u)})$$
for each $u\in U_\alpha^*$.
\end{thmenum}
\end{theorem}

\begin{proof}
See Sections~10 and~11 of \cite{tits-como}.
\end{proof}

Note that under the hypotheses of \ref{abc20}, we 
always identify each root group $U_\alpha$ with its image under the 
injection in \ref{abc20}.i. Thus, in particular, the $u$ in \ref{abc20}.ii
is really the canonical image in ${\rm Aut}(X,{\mathcal A})$ of an element 
$u\in U_\alpha$.

With the following definition we describe those non-discrete Euclidean buildings which 
were studied in \cite{bruh-tits}.

\begin{definition}\label{abc21}
A {\it Bruhat-Tits space} is a non-discrete Euclidean building
$(X,{\mathcal A})$ such that the following hold:
\begin{thmenum}
\item The spherical building $\Delta:=(X,{\mathcal A})^\infty$ is 
Moufang (in the sense of \ref{abc2} or \ref{abc10}).
\item The conclusions of \ref{abc20} hold.
\end{thmenum}
Let $(X,{\mathcal A})$ be a non-discrete Euclidean building satisfying (i) 
and let $\ell$ denote the dimension of $(X,{\mathcal A})$.
If $\ell=1$, saying that $\Delta$ is Moufang 
means we have a particular Moufang structure on $\Delta$ in mind,
and (ii) is to be interpreted with respect to this Moufang structure.
If $\ell\ge2$, then (ii) holds automatically. 
\end{definition}

\noindent
A non-discrete Euclidean building of rank $\ell\ge3$ always satisfies \ref{abc21}.i.
(This is proved, for example, \cite[40.3]{TW}.)
Thus in dimension three or higher, ``non-discrete Euclidean building'' and 
``Bruhat-Tits space'' are the same thing. 

\begin{convention}\label{abc24}
Let $\Delta$ be a Moufang building of rank one in the sense of \ref{abc10}.
When we say that ``$\Delta$ is the building at infinity of the
Bruhat-Tits space $(X,{\mathcal A})$,''
we mean that the conclusions of \ref{abc20} hold with respect to
the particular Moufang structure on $\Delta$ we have in mind.
\end{convention}

The following results of Bruhat-Tits and Tits are fundamental.

\begin{theorem}\label{abc23}
Let $(X,{\mathcal A})$, $A$, $x$, $\Delta$, $\Sigma$ and $\varphi_\alpha$ for $\alpha\in\Phi$
be as in \ref{abc20}. Then 
$$\varphi:=\{\varphi_\alpha\mid\alpha\in\Phi\}$$ 
is a valuation of the root datum of $\Delta$ based at $\Sigma$.
Moreover, the valuation $\varphi$ is independent 
of the choice of the point $x$ in $A$ up to equipollence (as defined in \ref{abc30}).
\end{theorem}

\begin{proof}
This is the first part of Theorem~3 in \cite{tits-como}.
\end{proof}

\begin{theorem}\label{abc31}
Let $\Delta$ be an irreducible spherical building of type $\Phi$ satisfying the Moufang condition,
let $\Sigma$ be an apartment of $\Delta$ (to which \ref{abc70} is applied)
and let $\varphi$ be a valuation of the 
root datum of $\Delta$ based at $\Sigma$ as defined in \ref{abc14}. Then there exists a 
Bruhat-Tits space $(X,{\mathcal A})$ of type $\Phi$, an apartment $A$ of $(X,{\mathcal A})$
and a point $x_A$ of $A$ such that the following hold:
\begin{thmenum}
\item $\Delta$ is the building at infinity of $(X,{\mathcal A})^\infty$ (in the sense
of \ref{abc24} if the rank of $\Delta$ is one)
and $\Sigma=A^\infty$.
\item For each $\alpha\in\Phi$, $\varphi_\alpha$ is the map which appears in \ref{abc20}.ii
when \ref{abc20} is applied to the triple $(X,{\mathcal A})$, $A$ and $x_A$.
\end{thmenum}
If $(X',{\mathcal A'})$, $A'$ and $x'_A$ is a second triple with these properties,
then there exists an isomorphism from $(X,{\mathcal A})$ to $(X',{\mathcal A}')$
mapping $A$ to $A'$ and $x_A$ to $x'_A$. 
\end{theorem}

\begin{proof}
Existence is proved in Section~7.4 of \cite{bruh-tits} and uniqueness in Proposition~6 of
\cite{tits-como}.
\end{proof}

\section{The Ree and Suzuki groups}
In this section we collect a few well known facts about the Ree and Suzuki 
groups. All of these results are contained more or less explicitly in \cite{tits-ree}
and \cite{tits-oct}.

There are three families of Ree and Suzuki groups. 
Beginning in \ref{abc33} and for the rest of this paper, we will refer to 
three cases which we will call ``case ${\sf B}$,'' ``case ${\sf F}$''
and ``case ${\sf G}$.'' 

\begin{notation}\label{abc33}
Let $K$ be a field of characteristic positive $p$ 
and suppose that $\theta$ is a {\it Tits endomorphism} 
of $K$. This means that $\theta$ is an endomorphism of $K$ such
that $\theta^2$ is the Frobenius map
$x\mapsto x^p$. Thus, in particular, 
$F:=K^\theta$ is a subfield of $K$ isomorphic to $K$ which contains
the subfield $K^p$. Suppose, too, that $p=2$
in cases ${\sf B}$ and ${\sf F}$ and $p=3$ in case ${\sf G}$. 
In case ${\sf B}$ let $L$ be an additive subgroup of $K$
containing $F$ such that $L\cdot F\subset L$ (so 
$L$ is a vector space over $F$) and $K=\langle L\rangle$
(where $\langle L\rangle$ denotes the subring of $K$ generated by $L$) and 
let $\Lambda$ denote the indifferent set $(K,L,L^\theta)$ as defined in
\cite[10.1]{TW}. 
In case ${\sf F}$ let
$\Lambda$ denote the composition algebra $(K,F)$ as defined in \cite[30.17]{weiss}.
In case ${\sf G}$ let $\Lambda$ denote 
the hexagonal system $(K/F)^\circ$ as defined
in \cite[15.20]{TW}. 
Let $\Delta$ denote the building ${\sf B}_2^{\mathcal D}(\Lambda)$
in case ${\sf B}$, the building ${\sf F}_4(\Lambda)$ in case ${\sf F}$ and
the building ${\sf G}_2(\Lambda)$
in case ${\sf G}$ in the notation described in \cite[30.17]{weiss}. Thus
${\sf B}_2^{\mathcal D}(\Lambda)$ is the Moufang quadrangle
called ${\mathcal Q}_{\mathcal D}(\Lambda)$ in \cite[16.4]{TW} 
and ${\sf G}_2(\Lambda)$ the Moufang hexagon called ${\mathcal H}(\Lambda)$ in 
\cite[16.8]{TW}. The type $\Phi$ of the building $\Delta$ 
is ${\sf B}_2$ in case ${\sf B}$, ${\sf F}_4$ in 
case ${\sf F}$ and ${\sf G}_2$ in case ${\sf G}$. (Alternatively, we can define
$\Delta$ to be the unique building of type $\Phi$ whose root datum 
is as described in \ref{abc6} below.)
\end{notation}

The building $\Delta$ in \ref{abc33} is split if and only if $K$ is perfect;
if $K$ is not perfect, $\Delta$ is simply a building of mixed type
(as defined, for example, in \cite[30.24]{weiss}). 

The following element $\tau$ plays a central role from now on.

\begin{notation}\label{abc66}
Let $\Phi$ be as in \ref{abc33}, let ${\mathbb A}$ be the ambient
space of $\Phi$ and let $S$ be a sector of $\Phi$.
There is a unique non-trivial element of ${\rm Aut}(\Phi)$ (as defined in \ref{abc0})
fixing $S$. This automorphism induces a non-type-preserving automorphism
of the Coxeter complex $\Sigma(\Phi)$ and has order two. We denote it by $\tau$. 
\end{notation}

\begin{theorem}\label{abc6}
Let $\Delta$, $\Phi$, $K$, $L$, etc.~be as in \ref{abc33}, let
$\tau$ and $S$ be as in \ref{abc66}, let $\Sigma$ be an apartment of $\Delta$, 
let $C$ be a chamber of $\Sigma$, let $\psi$ be the unique special isomorphism
from $\Sigma(\Phi)$ to $\Sigma$ mapping $S$ to $C$ and let $\Phi$ be identified
with the set of roots of $\Sigma$ via $\psi$ as indicated in \ref{abc70}.
Then for each $\alpha\in\Phi$, there exists an isomorphism $x_\alpha$ from the
additive group of $L$ in case ${\sf B}$, respectively, the additive group of $K$
in cases ${\sf F}$ and ${\sf G}$,
to the root group $U_\alpha$ of $\Delta$ such that the 
collection $(x_\alpha)_{\alpha\in\Phi}$ has the following properties:
\begin{thmenum}
\item There exists a unique automorphism $\rho$ of 
$\Delta$ mapping the pair $(C,\Sigma)$ to itself such that 
$$x_\alpha(t)^\rho=x_{\tau(\alpha)}(t)$$
for all $\alpha\in\Phi$ and all $t\in K$ (or all $t\in L$).
\item In cases ${\sf B}$ and ${\sf F}$, 
$$[U_\alpha,U_\beta]=1$$
whenever $\angle(\alpha,\beta)\le90^\circ$,
$$[x_\alpha(s),x_\beta(t)]=x_{\alpha+\beta}(st)$$
for all $s,t\in K$ whenever $\angle(\alpha,\beta)=120^\circ$ and 
$$[x_\alpha(s),x_\beta(t)]=x_{\sqrt{2}\alpha+\beta}(s^\theta t)
x_{\alpha+\sqrt{2}\beta}(st^\theta)$$
for all $s,t\in K$ (or all $s,t\in L$) whenever $\angle(\alpha,\beta)=135^\circ$.
\item In case ${\sf G}$, there exists parameters $\epsilon,\epsilon_1,\ldots,\epsilon_4$ 
such that
$$[U_\alpha,U_\beta]=1$$
whenever $\angle(\alpha,\beta)\le90^\circ$, 
$$[x_\alpha(s),x_\beta(t)]=x_{\alpha+\beta}(\epsilon st)$$
for all $s,t\in K$ whenever $\angle(\alpha,\beta)=120^\circ$ and 
\begin{align*}
[x_\alpha(s),x_\beta(t)]=x_{\sqrt{3}\alpha+\beta}(\epsilon_1s^\theta t)
x_{2\alpha+\sqrt{3}\beta}(&\epsilon_2s^2t^\theta)\cdot\\
&\cdot x_{\sqrt{3}\alpha+2\beta}(\epsilon_3s^\theta t^2)
x_{\alpha+\sqrt{3}\beta}(\epsilon_4st^\theta)
\end{align*}
for all $s,t\in K$ whenever $\angle(\alpha,\beta)=150^\circ$.
The parameters $\epsilon,\epsilon_1,\ldots,\epsilon_4$ are 
all equal to $+1$ or $-1$; their values 
depend on the ordered pair $(\alpha,\beta)$ but not on $s$ or $t$;
and their values are as in \ref{abc6a}--\ref{abc6c} 
if $\alpha,\beta$ both contain the chamber $C$.
\item In all three cases,
$$x_\beta(t)^{m_\Sigma(x_\alpha(1))}=x_{s_\alpha(\beta)}(\pm t)$$
for all $\alpha,\beta\in\Phi$ and all $t\in K$ (or $L$).
\end{thmenum}
\end{theorem}

\begin{proof}
Suppose first that we are in case ${\sf G}$ and let 
the roots of $\Phi$ be numbered $\alpha_1,\ldots,\alpha_{12}$ going around
the origin clockwise, where the indices are to be read modulo $12$. 
The set of roots of $\Sigma$ can be
identified with $\Phi$ so that $\alpha_i$ contains the chamber $C$ if and only if
$i\in[1,6]$. Thus $\tau(\alpha_i)=\alpha_{7-i}$ for all $i$.
By \cite[16.8]{TW}, $[U_{\alpha_i},U_{\alpha_j}]=1$ whenever $i,j\in[1,6]$ and $i<j\le i+3$
and there exist isomorphisms $x_i$ for each $i\in[1,6]$ 
from the additive group of $K$ to the root group $U_{\alpha_i}$
such that for all $s,t\in K$,
\begin{equation}\label{abc6a}
[x_1(s),x_6(t)]=x_2(-s^\theta t)x_3(-s^2t^\theta)x_4(s^\theta t^2)x_5(st^\theta),
\end{equation}
as well as
\begin{equation}\label{abc6b}
[x_1(s),x_5(t)]=x_3(-st)
\end{equation}
and 
\begin{equation}\label{abc6c}
[x_2(s),x_6(t)]=x_4(st).
\end{equation}
By \cite[7.5]{TW}, there exists a unique automorphism $\rho$ of $\Delta$
mapping $(C,\Sigma)$ to itself such that 
$$x_i(t)^\rho=x_{7-i}(t)$$
for all $i\in[1,6]$ and for all $t\in K$. Let $H$ be the chamberwise stabilizer 
of $\Sigma$ in ${\rm Aut}(\Delta)$ 
and let $Q$ be the subgroup of $H$ consisting of those elements $g$
such that for each root $\alpha$ of $\Sigma$, either $g$ centralizes $U_\alpha$
or $g$ inverts every element of $U_\alpha$. Let 
$$m_i=m_\Sigma(x_i(1))$$
for $i=1$ and $6$ and let 
$$N=\langle m_1,m_6\rangle.$$
By \cite[2.9(6)]{tits-oct} and \cite[29.12]{TW}, 
$|Q|=4$ and the kernel of the action of $N$ on $\Sigma$ is $N\cap Q$ and 
$N/N\cap Q\cong D_{12}$. For each $i\in[7,12]$, there thus is a unique
shortest word $g$ in $m_1$ and $m_6$ mapping $\alpha_j$ to $\alpha_i$, where
$j=1$ if $i$ is odd and $j=6$ if $i$ is even.
We set $x_i(t)=x_j(t)^g$ for all $t\in K$. 
By \cite[6.2]{TW}, $m_1^\rho=m_6$ and $m_6^\rho=m_1$.
It follows that $x_i(t)^\rho=x_{\tau(i)}(t)$ for all $t\in K$ and for all $i\in[7,12]$.
Thus (i) holds.
Let $\alpha,\beta\in\Phi$. By \cite[6.4]{TW} (or \cite[2.9]{tits-oct}), there exists $g\in N$
mapping $\alpha$ to $\beta$ 
such that $x_\alpha(t)^g=x_\beta(\pm t)$ for all $t\in K^*$. Since $g$ 
is unique up to an element of $N\cap Q$, it follows that 
$x_\alpha(t)^g=x_\beta(\pm t)$ for all $t\in K^*$ and for all $g\in N$ mapping
$\alpha$ to $\beta$. Thus (iv) holds (in case ${\sf G}$). By \ref{abc6a}--\ref{abc6c}, it follows
that also (iii) holds.

The proof in case ${\sf B}$ is virtually the same, except that we cite
\cite[16.4]{TW} in place of \cite[16.8]{TW} and use the observation that 
this time $m_\Sigma(u)^2=1$
for all $\alpha\in\Phi$ and all $u\in U_\alpha^*$ by \ref{abc8a}.ii (since all the root groups
are of exponent two). In case ${\sf F}$, the proof in the 
previous two cases can be suitably modified (using \cite[32.14]{weiss} 
in place of \cite[16.4]{TW}). We do not give the details, however,
since the claims in case ${\sf F}$ also follow from 
the observations in \cite[1.1--1.2]{tits-oct}.
\end{proof}

\begin{definition}\label{abc7}
Let $\Delta$ and $\rho$ be as in \ref{abc6}, 
let $\Delta^\rho$ denote the set of chambers fixed by $\rho$ and let 
$G^\dagger$ be as in \ref{abc13}. 
Let $G$ be the group induced on $\Delta^\rho$ by the centralizer of $\rho$ in 
$G^\dagger$. The group $G$ is commonly called ${\rm Suz}(K,\theta,L)$
in case ${\sf B}$ (these are the {\it Suzuki groups}), 
${\rm Ree}(K,\theta)$ in case ${\sf G}$ 
(these are the {\it Ree groups}) and $^2\!F_4(K,\theta)$ in case
${\sf F}$. 
The groups $^2\!F_4(K,\theta)$ were also discovered by Ree and are also 
sometimes called Ree groups, 
but these groups are now more commonly associated
with Tits due to his thorough study of them in \cite{tits-oct}. The name
${\rm Suz}(K,\theta,L)$ is usually abbreviated to ${\rm Suz}(K,\theta)$ 
when $L=K$.
\end{definition}

We use the rest of this section to prove a result (\ref{abc80}) about 
valuations of the root datum of the building $\Delta$.

\begin{proposition}\label{abc78}
Let $\Delta$, $\Phi$, $\Sigma$ and $(x_\alpha)_{\alpha\in\Phi}$ be as in \ref{abc6},
let $\alpha,\beta\in\Phi$, let $t\in K$, let $u\in K^*$ (or $t\in L$ and $u\in L^*$
in case ${\sf B}$) and let
$$h(u)=m_\Sigma(x_\alpha(1))m_\Sigma(x_\alpha(u)).$$
Then the following hold:
\begin{thmenum}
\item In cases ${\sf B}$ and ${\sf F}$,
$$x_\beta(t)^{h(u)}=x_\beta(u^{-\theta}t)$$
if $\angle(\alpha,\beta)=45^\circ$ and
$$x_\beta(t)^{h(u)}=x_\beta(u^\theta t)$$
if $\angle(\alpha,\beta)=135^\circ$.
\item In case ${\sf G}$,
$$x_\beta(t)^{h(u)}=x_\beta(u^\theta t)$$
if $\angle(\alpha,\beta)=30^\circ$ and 
$$x_\beta(t)^{h(u)}=x_\beta(u^{-\theta}t)$$
if $\angle(\alpha,\beta)=150^\circ$.
\item In all three cases,
$$x_\beta(t)^{h(u)}=x_\beta(u^{-2}t)$$
if $\alpha=\beta$,
$$x_\beta(t)^{h(u)}=x_\beta(u^{-1}t)$$
if $\angle(\alpha,\beta)=60^\circ$,
$$x_\beta(t)^{h(u)}=x_\beta(t)$$
if $\angle(\alpha,\beta)=90^\circ$,
$$x_\beta(t)^{h(u)}=x_\beta(ut)$$
if $\angle(\alpha,\beta)=120^\circ$ and
$$x_\beta(t)^{h(u)}=x_\beta(u^2t)$$
if $\alpha=-\beta$.
\end{thmenum}
\end{proposition}

\begin{proof}
By \ref{abc8a}.i, $h(u)$ acts trivially on $\Sigma$ and hence normalizes $U_\beta$.
The claims hold for $\beta\ne\pm\alpha$ by parts (ii) and (iii) of \ref{abc6} 
and \cite[2.9]{tits-oct}. 

Suppose that we can choose $\beta\in\Phi$ such that $\angle(\alpha,\beta)=120^\circ$.
By \ref{abc6}, we have
\begin{equation}\label{abc78a}
[x_\alpha(t),x_\beta(s)]=x_{\alpha+\beta}(\epsilon st)
\end{equation}
and 
\begin{equation}\label{abc78c}
[x_{\alpha+\beta}(s),x_{-\alpha}(t)]=x_\beta(\epsilon' st)
\end{equation}
for all $s\in K$, where $\epsilon,\epsilon'=\pm1$.
We know that $x_\beta(1))^{h(u)}=x_\beta(u)$ and 
$$x_{\alpha+\beta}(t)^{h(u)}=x_{\alpha+\beta}(u^{-1}t).$$
Setting $s=1$ in \ref{abc78a} and conjugating by $h(u)$, we obtain
$$[x_\alpha(t)^{h(u)},x_\beta(u)]=x_{\alpha+\beta}(\epsilon u^{-1}t).$$
Comparing this identity with \ref{abc78a} itself, 
we conclude that $x_\alpha(t)^{h(u)}=x_\alpha(u^{-2}t)$.
Setting $s=1$ in \ref{abc78c} and conjugating by $h(u)$, we conclude
similarly that $x_{-\alpha}(t)^{h(u)}=x_{-\alpha}(u^2t)$.
  
Suppose next that there is no $\beta\in\Phi$ such that $\angle(\alpha,\beta)=120^\circ$.
Then we are in case ${\sf B}$, so ${\rm char}(K)=2$, and 
we can find $\beta$ such that $\angle(\alpha,\beta)=135^\circ$. Thus
\begin{equation}\label{abc78b}
[x_\alpha(t),x_\beta(s)]=x_{\sqrt{2}\alpha+\beta}(t^\theta s)x_{\alpha+\sqrt{2}\beta}(ts^\theta)
\end{equation}
and
\begin{equation}\label{abc78d}
[x_{\sqrt{2}\alpha+\beta}(s),x_{-\alpha}(t)]=x_{\alpha+\sqrt{2}\beta}(s^\theta t)x_\beta(st^\theta)
\end{equation}
by \ref{abc6}.ii. We know that 
$h(u)$ centralizes $U_3$, $x_\beta(1)^{h(u)}=x_\beta(t^\theta)$ and
$$x_{\sqrt{2}\alpha+\beta}(t^\theta)^{h(u)}=x_{\sqrt{2}\alpha+\beta}(u^{-\theta}t^\theta).$$
Setting $s=1$ in \ref{abc78b} and conjugating by $h(u)$, we obtain
$$[x_\alpha(t)^{h(u)},x_\beta(u^\theta)]=
x_{\sqrt{2}\alpha+\beta}(u^{-\theta}t^\theta)x_{\alpha+\sqrt{2}\beta}(t).$$
Comparing this identity with 
\ref{abc78b} itself, we conclude again that 
$$x_\alpha(t)^{h(u)}=x_\alpha(u^{-2}t).$$
Setting $s=1$ in the identity \ref{abc78d} and conjugating 
by $h(u)$, we conclude similarly that $x_{-\alpha}(t)^{h(u)}=x_{-\alpha}(u^2t)$
also in this case.
\end{proof}

\begin{corollary}\label{abc81}
Let $\Delta$, $\Phi$, $\Sigma$ and $(x_\alpha)_{\alpha\in\Phi}$ be as in \ref{abc6},
let $G^\dagger$ be as in \ref{abc13} and let $H$ be the chamberwise stabilizer of
$\Sigma$ in $G^\dagger$. 
Then for each $\alpha\in\Phi$ and for each $h\in H$, there exists $z\in K^*$
(or $z\in F^*$ in case ${\sf B}$, where $F=K^\theta$) such that
$$x_\alpha(t)^h=x_\alpha(zt)$$
for all $t\in K$ (or all $t\in L$).
\end{corollary}

\begin{proof}
By \cite[33.9]{TW}, 
$$H=\langle m_\Sigma(x_\alpha(1))m_\Sigma(x_\alpha(u))\mid\alpha\in\Phi\text{ and }u\in K^*
\text{ (or $L^*$)}\rangle.$$
The claim holds, therefore, by \ref{abc78}.
\end{proof}

\begin{notation}\label{abc79}
Let $\Phi$ be as in \ref{abc33}. Let $\Phi=\Phi_0\cup\Phi_1$ be the partition
of $\Phi$ into two subsets $\Phi_0$ and $\Phi_1$ 
such that two elements of $\Phi$ are in the same
subset if and only if they have the same length.
\end{notation}

\begin{proposition}\label{abc80}
Let $\Delta$, $\Phi$, $\Sigma$, $\rho$ and $(x_\alpha)_{\alpha\in\Phi}$ be as in \ref{abc6},
let $\Phi_0$ and $\Phi_1$ be as in \ref{abc79},
let $\nu$ be a (real valued) valuation of the field $K$
and let $p={\rm char}(K)$. Let 
$$\varphi_\alpha(x_\alpha(t))=\nu(t)$$
for all $\alpha\in\Phi_0$ and all $t\in K^*$ (or all $t\in L^*$), let 
$$\varphi_\alpha(x_\alpha(t))=\nu(t^\theta)/\sqrt{p}$$
for all $\alpha\in\Phi_1$ and all $t\in K^*$ (or all $t\in L^*$) and let 
$$\varphi=(\varphi_\alpha)_{\alpha\in\Phi}.$$ 
Then the following hold:
\begin{thmenum}
\item $\varphi$ is a valuation of the root datum of $\Delta$ based at $\Sigma$.
\item $\varphi$ is $\rho$-invariant 
(as defined in \ref{abc32}) if and only if $\nu$ is $\theta$-invariant. 
\end{thmenum}
\end{proposition}

\begin{proof}
The collection $\varphi$ satisfies condition (V1) in \ref{abc14} by the definition of a valuation.
By parts (ii) and (iii) of \ref{abc6} and some calculation, $\varphi$ satisfies condition
(V2). Choose $\alpha,\beta\in\Phi$, $u\in U_\alpha^*$ and $g\in U_\beta^*$.
Suppose first that $\alpha=\beta$. By \ref{abc6}.iv, 
$$\varphi_{-\alpha}(g^{m_\Sigma(u)})
=\varphi_\alpha(g^{m_\Sigma(u)m_\Sigma(x_\alpha(1))^{-1}}).$$
By \ref{abc8a}.ii, therefore
$$\varphi_{-\alpha}(g^{m_\Sigma(u)})
=\varphi_\alpha(g^{(m_\Sigma(x_\alpha(1))m_\Sigma(-u))^{-1}}).$$
Hence 
$$\varphi_{-\alpha}(g^{m_\Sigma(u)})-\varphi_\alpha(g)=-2\varphi_{\alpha}(u)$$
by \ref{abc78}.iii. By \ref{abc6}.iv
and \ref{abc78} (and a similar calculation), the quantity
$$\varphi_{s_\alpha(\beta)}(g^{m_\Sigma(u)})-\varphi_\beta(g)$$
is independent of $g$ whenever $\beta\ne\alpha$. Thus $\varphi$ satisfies
condition (V3). Hence (i) holds.
By \ref{abc6}.i, $\varphi$ is $\rho$-invariant if and only if $\nu(t^\theta)/\sqrt{p}
=\nu(t)$ for all $t\in K^*$. Thus (ii) holds.
\end{proof}

\section{Spherical Buildings for the Ree and Suzuki groups}

In this section we show that the $G$-set $\Delta^\rho$ defined in \ref{abc7} 
has the structure of a spherical building satisfying
the Moufang condition whose root groups generate the group $G$.

\begin{notation}\label{abc400}
Let $\dot{\Phi}$ and $\dot{\mathbb A}$
be the sets called $\Phi$ and ${\mathbb A}$ in \ref{abc33} and \ref{abc66}.
We then set ${\mathbb A}$ equal to the set of fixed points of
the automorphism $\tau$ of $\dot{\Phi}$ introduced in \ref{abc66} (which
we continue to call $\tau$.
\end{notation}

\begin{proposition}\label{abc98}
Let $\dot{\mathbb A}$, $\dot{\Phi}$, $\tau$ and ${\mathbb A}$ be as in \ref{abc400} and 
let $\ddot{\alpha}=\dot{\alpha}+\dot{\alpha}^\tau$ (so $\ddot{\alpha}\in{\mathbb A}$
since $\tau^2=1$) for each $\dot{\alpha}\in\dot{\mathbb A}$.
Then the following hold:
\begin{thmenum}
\item If $\dot{\alpha}\in\dot{\Phi}$, then $\ddot{\alpha}\ne0$.
\item Let 
$$\Phi=\{\ddot{\alpha}/|\ddot{\alpha}|\mid\dot{\alpha}\in\dot{\Phi}\}.$$
Then $\Phi$ is, up to an isometry of ${\mathbb A}$, the root system ${\sf A}_1$
in cases ${\sf B}$ and ${\sf G}$ and
${\sf I}_2(8)$ (as defined in \ref{abc0}) in case ${\sf F}$.
\item For each $\dot{\alpha}\in\dot{\Phi}$, the half-space
$$H_\alpha:=\{v\in{\mathbb A}\mid v\cdot\ddot{\alpha}\ge0\}$$ 
of $\Sigma(\Phi)$ (as defined in \ref{abc60}) is the intersection
of the half-space
$$H_{\dot{\alpha}}:=\{v\in\dot{\mathbb A}\mid v\cdot\dot\alpha\ge0\}$$
of $\Sigma(\dot{\Phi})$ with ${\mathbb A}$. 
\item The map $\dot{F}\mapsto\dot{F}\cap{\mathbb A}$ is an
inclusion-preserving 
bijection from the set of $\tau$-invariant faces
of $\dot{\Phi}$ to the set of faces of $\Phi$.
\end{thmenum}
\end{proposition}

\begin{proof}
In cases ${\sf B}$ and ${\sf G}$, the dimension of $\dot{\mathbb A}$ is only two;
we leave it to the reader to verify all the claims in these two cases.
In case ${\sf F}$, these assertions are proved in Section~1.3 of \cite{tits-oct}.
\end{proof}

\begin{lemma}\label{abc29x}
Let $\dot{\Phi}$ and $\tau$ be as in \ref{abc400}, let 
$\Phi$ be as in \ref{abc98}.ii,
let $\dot{\Delta}$ be an arbitrary building of type $\dot{\Phi}$ which is not 
necessarily thick, let $\rho$ be a non-type-preserving 
automorphism of $\dot{\Delta}$ of order two
and let $\dot{\Sigma}$ be an apartment of $\dot{\Delta}$. 
Then the following are equivalent:
\begin{thmenum}
\item $\dot{\psi}\circ\tau=\rho\circ\dot{\psi}$
for some type-preserving isomorphism $\dot{\psi}$ from $\Sigma(\dot{\Phi})$ to $\dot{\Sigma}$. 
\item $\rho$ fixes two opposite chambers of $\dot{\Sigma}$.
\item $\rho$ maps $\dot{\Sigma}$ to itself and fixes at least one
chamber of $\dot{\Sigma}$.
\end{thmenum}
\end{lemma}

\begin{proof}
Let $S$ be as in \ref{abc66}. We think of $S$ as a chamber of $\Sigma(\dot{\Phi})$
and let $S'$ be the unique opposite chamber of $\Sigma(\dot{\Phi})$. 
The map $\tau$ fixes both $S$ and $S'$. Hence if $\dot{\psi}$ is as in (i),  
then $\rho$ fixes the two opposite chambers $\dot{\psi}(S)$ and 
$\dot{\psi}(S')$ of $\dot{\Sigma}$. Thus (i) implies (ii). Since opposite
chambers are contained in a unique apartment, (ii) implies (iii). 
Now suppose (iii) holds. Let $C$ be a chamber of $\dot{\Sigma}$ fixed by 
$\rho$ and let $\dot{\psi}$ be the unique isomorphism from $\Sigma(\dot{\Phi})$
to $\dot{\Sigma}$ mapping $S$ to $C$. Since the composition $\rho^{-1}\circ\dot{\psi}\circ\tau$
is also a type-preserving automorphism from $\Sigma(\dot{\Phi})$
to $\dot{\Sigma}$ mapping $S$ to $C$, it must equal $\dot{\psi}$. Thus (i) holds.
\end{proof}

\begin{definition}\label{abc29y}
Under the hypotheses of \ref{abc29x}, we call an apartment of $\dot{\Delta}$
{\it $\rho$-compatible} if it satisfies the three equivalent conditions in
\ref{abc29x}.
\end{definition}

In the following result, we make implicit use of the fact that a building $\Delta$
is completely determined by the graph whose vertices are the chambers of $\Delta$,
where two chambers are joined by an edge whenever there is a panel of $\Delta$ 
containing them both. (This is the point of view, for example, in \cite{me-sph}.)

\begin{theorem}\label{abc29}
Let $\dot{\Phi}$ and $\tau$ be as in \ref{abc400}, let 
$\Phi$ be as in \ref{abc98}.ii,
let $\dot{\Delta}$ be an arbitrary building of type $\dot{\Phi}$ which is not 
necessarily thick, let $\rho$ be a non-type-preserving 
automorphism of $\dot{\Delta}$ of order two
and let $\dot{\Sigma}$ be an apartment of $\dot{\Delta}$. Suppose that
$\dot{\Sigma}$ is $\rho$-compatible as defined in \ref{abc29y}.
Let $\dot{\psi}$ be as in \ref{abc29x} and let $\dot{\pi}$
be the bijection from $\dot{\Phi}$ to the set of roots of $\dot{\Sigma}$ induced
by the isomorphism $\dot{\psi}$. Let $\Delta$ be the graph whose
vertices are the chambers of $\dot{\Delta}$ fixed by $\rho$, where two such chambers
are joined by an edge whenever they are opposite in a residue of rank
two fixed by $\rho$ and let $\Sigma$ be the subgraph of $\Delta$ 
spanned by the chambers of $\dot{\Sigma}$ fixed by $\rho$.
Then the following hold:
\begin{thmenum}
\item There is a unique isomorphism $\psi$ from $\Sigma(\Phi)$ to $\Sigma$ 
such that 
$$\psi(\dot{S}\cap{\mathbb A})=\dot{\psi}(\dot{S})$$
for each $\tau$-invariant sector $\dot{S}$ of $\dot{\Phi}$.
\item If $\pi$ is the bijection from $\Phi$ to the set of roots of $\Sigma$ induced
by $\psi$, then 
$$\pi(\ddot{\alpha}/|\ddot{\alpha}|)=\dot{\pi}(\dot{\alpha})\cap\Sigma$$ 
for all $\dot{\alpha}\in\dot{\Phi}$,
where the map $\dot{\alpha}\mapsto\ddot{\alpha}$ is as in \ref{abc98}. 
\item $\Delta$ is a building of type $\Phi$ whose apartments are 
the subgraphs spanned by the chambers
fixed by $\rho$ in $\rho$-compatible apartments of $\dot{\Delta}$.
\end{thmenum}
\end{theorem}

\begin{proof}
Assertions (i) and (ii) hold by \ref{abc98}. Next we observe that 
if $R$ is the $\rho$-invariant residue of rank two containing two 
adjacent chambers of $\Delta$, then all the 
chambers in $R$ fixed by $\rho$ are pairwise opposite in $R$
and hence adjacent in $\Delta$. This means
that $\Delta$ is a {\it chamber system} as defined at the beginning
of \cite{ronan}. To show that (iii) holds,
it therefore suffices, by Theorem~3.11 in \cite{ronan} and Theorem~8.21 in \cite{me-sph},
to show that every two chambers of $\Delta$ are contained in an
apartment of $\dot{\Delta}$ containing opposite chambers fixed by $\rho$.
In cases ${\sf B}$ and ${\sf G}$, every two chambers of $\Delta$ are 
opposite in $\dot{\Delta}$. We can thus assume that we are in case ${\sf F}$.

Let $\gamma=(C_0,C_1,\ldots,C_k)$ be a gallery in $\Delta$.
Then for each $i\in[1,k]$, the chambers $C_{i-1}$ and 
$C_i$ are opposite in a $\rho$-invariant $J_i$-residue $R_i$,
where $J_i$ is a two-element subset of the vertex
set $I$ of the diagram ${\sf F}_4$ invariant under the non-trivial
automorphism of this diagram. We will say that $\gamma$ is
{\it alternating} if $J_i\ne J_{i-1}$ for all $i\in[2,k]$. Let 
$\dot{\gamma}$ be a gallery in $\dot{\Delta}$ from
$C_0$ to $C_k$ containing $C_i$ also for all $i\in[1,k-1]$
such that the subgallery from $C_{i-1}$ to $C_i$ is a minimal
gallery in $R_i$ for each $i\in[1,k]$. Thus the length 
of $\dot{\gamma}$ is 
$$m:=d_1+d_2+\cdots+d_k,$$
where $d_i$ is the diameter of $R_i$ for all $i\in[1,k]$.
Now suppose that $k=8$. Then $m$ equals the diameter of $\dot{\Sigma}$
and if $D$ and $D'$ are opposite chambers in $\dot{\Sigma}\cap\Delta$,
then there exists a unique minimal gallery from $D$ to $D'$ 
that has the same type as $\dot{\gamma}$. By Proposition~7.7.ii in \cite{me-sph},
it follows that $\dot{\gamma}$ is minimal and hence $C_0$ and $C_8$
are opposite. By Corollaries~8.6 and ~8.9 in \cite{me-sph}, therefore,
$\dot{\gamma}$ is contained in an apartment of $\dot{\Delta}$.
It thus suffices to show that 
every alternating gallery of arbitrary length $k$
in $\Delta$ can be extended to an alternating gallery in $\Delta$ of length $k+1$
and that every two chambers of $\Delta$ are 
joined by an alternating gallery in $\Delta$.

Suppose first that the gallery 
$\gamma=(C_0,C_1,\ldots,C_k)$ is alternating and let $J_i$ for $i\in[1,k]$
be as in the previous paragraph.
Let $J=I\backslash J_k$ and let $R$ be the unique $J$-residue
containing $C_k$. Since $\rho$ fixes $C_k$, the residue $R$ is also 
$\rho$-invariant. The residue $R$ is a generalized $n$-gon for $n=2$
or $4$. Let $\gamma$ be a gallery in $R$ of length $n/2$ starting at $C_k$ and let
$\gamma_1$ be the concatenation of $\gamma^{-1}$ with $\gamma^\rho$.
Then $\gamma_1$ is a minimal galley of length $n$ (because
$\rho$ is not type-preserving) and $\rho$ preserves
$\gamma_1$ (because $\rho^2=1$).
Therefore $\rho$ fixes the unique chamber $C_{k+1}$ opposite $C_k$
in the unique apartment of $R$ containing $\gamma_1$.
Thus $(C_0,\ldots,C_k,C_{k+1})$ is an alternating gallery extending
$\gamma$.

Now suppose that $C$ and $C'$ are two arbitrary chambers of $\Delta$
and let $e$ be the distance between them in $\dot{\Delta}$. 
Since
$\Delta$ is a chamber system (as observed above), 
we can obtain an alternating gallery from $C$ to $C'$ from an arbitrary
gallery from $C$ to $C'$ simply by discarding superfluous chambers.
It will suffice to show, therefore, 
that there is a gallery from $C$ to $C'$ in $\Delta$.
We proceed by induction with respect to $e$. We can suppose that 
$e>0$. Thus we can choose a chamber $C_1$ adjacent to $C'$ that is at distance
$e-1$ from $C$. Let $R$ be the unique $\rho$-invariant
residue of rank two containing both $C$ and $C_1$ and let
$C_2={\rm proj}_R\,C'$. Since $\rho$ fixes $C'$ and $R$,
it fixes $C_2$ too. Thus $C_2$ is a chamber of $\Delta$ adjacent
to $C$ in $\Delta$ and at distance strictly less than $e$ to
$C'$. By induction, we conclude that there is, in fact, a gallery in $\Delta$
from $C$ to $C'$.
\end{proof}

\begin{notation}\label{abc48}
Let $\Delta$, $\Sigma$, $\psi$, $\rho$ as well as $U_\alpha$ and $x_\alpha$ 
be as in \ref{abc6}. Note that the type of $\Delta$
is, according to \ref{abc400}, now $\dot{\Phi}$ rather than $\Phi$.
In order to focus on the set $\Delta^\rho$ defined in \ref{abc7}, we now replace 
also the designations $\Delta$, $\Sigma$, $\psi$, $U_\alpha$ and $x_\alpha$ 
by $\dot{\Delta}$, $\dot{\Sigma}$, $\dot{\psi}$,
$U_{\dot{\alpha}}$ and $x_{\dot{\alpha}}$ (but let $\rho$ remain $\rho$).
We then set $\Delta=\dot{\Delta}^\rho$ 
and $\Sigma=\dot{\Sigma}\cap\Delta$. 
By \ref{abc29} applied
to these data, $\Delta$ has (canonically) the structure of a building of type $\Phi$,
where $\Phi$ is as in \ref{abc98}.ii, and $\Sigma$ is an apartment of $\Delta$.
Let $\pi$ be the map obtained from these data in \ref{abc29}.ii 
and for each $\alpha\in\Phi$, let $\dot{\Phi}_\alpha$ denote the pre-image of
$\alpha$ under the surjection $\dot{\alpha}\mapsto\ddot{\alpha}/|\ddot{\alpha}|$
from $\dot{\Phi}$ to $\Phi$, where the map $\dot{\alpha}\mapsto\ddot{\alpha}$
is as in \ref{abc98}.
\end{notation}

\begin{theorem}\label{abc49}
Let $\dot{\Phi}$, $U_{\dot{\alpha}}$ for $\dot{\alpha}\in\dot{\Phi}$, $\Phi$,
$\dot{\Phi}_\alpha$ for $\alpha\in\Phi$,
$\Delta$, $\Sigma$, $\rho$ and $\pi$ be as in \ref{abc48} and
let $\Phi$ be identified with the set of roots of $\Sigma$ via $\pi$.
For each $\alpha\in\Phi$, let $U_\alpha$ denote the 
centralizer of $\rho$ in the subgroup 
$$\langle U_{\dot{\alpha}}\mid\dot{\alpha}\in\dot{\Phi}_\alpha\rangle$$
of ${\rm Aut}(\dot{\Delta})$. Then $U_\alpha$ acts faithfully on $\Delta$
for each $\alpha\in\Phi$ (in all three cases); $(\Sigma,(U_\alpha)_{\alpha\in\Phi})$
is a Moufang structure on $\Delta$ in cases ${\sf B}$ and ${\sf G}$; and in
case ${\sf F}$, $\Delta$ is Moufang and for each $\alpha\in\Phi$,
$U_\alpha$ is the corresponding root group.
\end{theorem}

\begin{proof}
Let $\dot{\Delta}$ be as in \ref{abc48} and let 
$\alpha\in\Phi$. We think of $\alpha$ as a root of $\Sigma$ and 
choose a panel of $\Sigma$ containing one chamber
$C$ in $\alpha$ and another $C'$ not in $\alpha$. 
Then $C$ and $C'$ are opposite in a unique rank two residue $\dot{R}$ of $\dot{\Delta}$
fixed by $\rho$. By \ref{abc29}.ii, $\dot{\Phi}_\alpha$ consists
of precisely those roots of $\dot{\Sigma}$ that contain
$C$ but not $C'$. By Proposition~8.13 in \cite{me-sph}, 
the map $\dot{\alpha}\mapsto\dot{\alpha}\cap\dot{R}$
is thus a bijection from $\dot{\Phi}_\alpha$ to the set of roots
of the apartment $\Sigma\cap\dot{R}$ of $\dot{R}$ and by Proposition~11.10 in \cite{me-sph},
$U_{\dot{\alpha}}$ induces the root group on $\dot{R}$ corresponding to
the root $\dot{\alpha}\cap\dot{R}$ for each $\dot{\alpha}\in\dot{\Phi}_\alpha$. 
By Theorem~11.11.ii in \cite{me-sph}, therefore, the group
$$\langle U_{\dot{\alpha}}\mid\dot{\alpha}\in\dot{\Phi}_\alpha\rangle$$
acts sharply transitively on the set $Q$ of chambers of $\dot{R}$
which are opposite $C$ in $\dot{R}$. 
Thus the group $U_\alpha$ acts sharply transitively on the
fixed point set $Q^\rho$ of $\rho$ in $Q$. 
It follows that $U_\alpha$ acts faithfully on $\Delta$ and
by Theorem~9.3 in \cite{me-sph}, that $U_\alpha$ acts sharply transitively on the set of
apartments of $\Delta$ containing $\alpha$. 
We conclude that $(U_\alpha)_{\alpha\in\Phi}$
is a Moufang structure on $\Delta$ (as defined in \ref{abc10})
in cases ${\sf B}$ and ${\sf G}$.

Now suppose that we are in case ${\sf F}$ and choose 
a panel of $\Delta$ containing two chambers $C_1$ and $C_1'$ in $\alpha$.
There exists a unique rank two residue $\dot{R}_1$ of $\dot{\Delta}$
fixed by $\rho$ containing $C_1$ and $C_1'$. Let $\dot{P}$
be a panel of $\dot{R}_1$ containing two chambers of $\dot{\Sigma}$.
If $\dot{\alpha}\in\dot{\Phi}_\alpha$, then $\dot{\alpha}$ contains
both $C_1$ and $C_1'$, hence the apartment $\dot{R}_1\cap\dot{\Sigma}$ of 
$\dot{R}_1$ is contained in $\dot{\alpha}$
(since roots are convex) and thus $U_{\dot{\alpha}}$ acts trivially both
on $\dot{R}_1\cap\dot{\Sigma}$ and on $\dot{P}$.
By Theorem~9.7 in \cite{me-sph},
therefore, $U_{\dot\alpha}$ acts trivially on $\dot{R}_1$ for all 
$\dot{\alpha}\in\dot{\Phi}_\alpha$. 
It follows that $U_\alpha$
is contained in the root group of $\Delta$ corresponding to
$\alpha$. Thus $\Delta$ is Moufang since 
$U_\alpha$ acts transitively (in fact, sharply transitively)
on the set of apartments of $\Delta$ containing $\alpha$ and $\alpha$ is
arbitrary. By Theorem~9.3 and Proposition~11.4 in 
\cite{me-sph}, the root group corresponding to $\alpha$ 
also acts sharply transitively on the set of apartments of $\Delta$ containing
$\alpha$. It follows that $U_\alpha$ equals this root group.
\end{proof}

\begin{convention}\label{abc97}
We will sometimes refer to a building $\Delta$ of the sort that appears in \ref{abc49}
as a {\it Suzuki-Ree building}. When we say that $\Delta$ is a Suzuki-Ree building
in cases ${\sf B}$ or ${\sf G}$, we mean that we have in mind the Moufang structure
on $\Delta$ described in \ref{abc49}.
\end{convention}

\begin{remark}\label{abc107}
Let $\Delta_{\sf F}$ and $\dot{\Delta}_{\sf F}$ be the buildings 
called $\Delta$ and $\dot{\Delta}$ in \ref{abc48} in 
case ${\sf F}$ and let $\Delta_{\sf B}$ and $\dot{\Delta}_{\sf B}$ 
denote the buildings called $\Delta$ and $\dot{\Delta}$ in \ref{abc48} in case ${\sf B}$ 
under the assumption that $L=K$, where $L$ is as in \ref{abc33}.
Then there exist residues of rank two of $\dot{\Delta}_{\sf F}$ fixed by 
$\rho$ which are isomorphic to the building $\dot{\Delta}_{\sf B}$.
Let $\dot{R}$ be one of these residues and let $P$
be the corresponding panel of $\Delta_{\sf F}$.
Then $P$ can be identified with the building $\Delta_{\sf B}$ in such a way that
the canonical Moufang structure on $P$ which comes from $\Delta_{\sf F}$
as described in \ref{abc19}
coincides with the Moufang structure on $\Delta_{\sf B}$ 
described in \ref{abc49}. 
\end{remark}

\begin{notation}\label{abc72}
Let $S=L\times L$ in case ${\sf B}$ and let $S=K\times K$ in case  ${\sf F}$.
In both of these two cases, let 
$$(s,t)\cdot(u,v)=(s+u,t+v+s^\theta u)$$
and 
$$R(s,t)=s^{\theta+2}+st+t^\theta$$
for all $(s,t)\in S$. In case ${\sf G}$, let
$T=K\times K\times K$, let 
$$(r,s,t)\cdot(w,u,v)=(r+w,s+u+r^\theta w,t+v-ru+sw-r^{\theta+1}w)$$
for all $(r,s,t),(w,u,v)\in T$ and let  
$$N(r,s,t)=r^{\theta+1}s^\theta-rt^\theta-r^{\theta+3}s-r^2s^2+s^{\theta+1}+t^2-r^{2\theta+4}$$
for all $(r,s,t)\in T$. Then $S$ and $T$ are groups (with multiplication $\cdot$),
$$(s,t)^{-1}=(s,t+s^{\theta+1})$$
for all $(s,t)\in S$ and 
$$(r,s,t)^{-1}=(-r,-s+r^{\theta+1},-t)$$
for all $(r,s,t)\in T$.
We will call $R$ the {\it norm of the group $S$} and $N$ the 
{\it norm of the group $T$}.
The center of $S$ is $\{(0,t)\mid t\in K\text{ (or $t\in L$)}\}$ and
the center of $T$ is $\{(0,0,t)\mid t\in K\}$; both centers are isomorphic
to the additive group of $K$.
\end{notation}

(Note that in case ${\sf B}$, the product
$K^\theta L$ is contained in $L$ (by \ref{abc33}) and thus the products $s^\theta t$,
$R(s,t)^{2-\theta}u$ and $R(s,t)^\theta v$ 
are contained in $L$ for all $(s,t),(u,v)\in S$ even though the norm $R(s,t)$ 
is not necessarily contained in $L$. See \ref{abc55}.i.)

\begin{remark}\label{abc750}
It is shown in \cite{tits-ree} that the maps $R$ and $N$ are anisotropic.
By this we mean that $R(s,t)=0$ only if $(s,t)=0$ and
$N(r,s,t)=0$ only if $(r,s,t)=0$.  
\end{remark}

\begin{proposition}\label{abc55}
Let $\Delta$, $\Phi$, $\Sigma$, the identification of $\Phi$ with the set of roots
of $\Sigma$ and the root groups $U_\alpha$ be as in \ref{abc49}, 
let $G$ be the corresponding Ree or Suzuki group as defined in \ref{abc7} and 
let the groups $S$ and $T$ and their norms $R$ and $N$ be as in \ref{abc72}.
Then there exist $\alpha\in\Phi$ such that the following hold:
\begin{thmenum}
\item In cases ${\sf B}$ and ${\sf F}$, $U_\alpha\cong S$ and
there exists an isomorphism $x_\alpha$ from $S$ to $U_\alpha$ that 
$$x_\alpha(u,v)^{h(s,t)}=x_\alpha(R(s,t)^{2-\theta}u,R(s,t)^\theta v)$$
for all $(u,v)\in S$ and all $(s,t)\in S^*$, where
$$h(s,t)=m_\Sigma(x_\alpha(0,1))m_\Sigma(x_\alpha(s,t)).$$
\item In case ${\sf G}$, $U_\alpha\cong T$ 
and there exists an isomorphism $x_\alpha$ from $T$ to $U_\alpha$ such that
$$x_\alpha(w,u,v)^{h(r,s,t)}
=x_\alpha(N(r,s,t)^{2-\theta}w,N(r,s,t)^{\theta-1}u,N(r,s,t)v)$$
for all $(w,u,v)\in T$ and all $(r,s,t)\in T^*$, where
$$h(r,s,t)=m_\Sigma(x_\alpha(0,0,1))m_\Sigma(x_\alpha(r,s,t)).$$
\end{thmenum}
The element $\alpha$ is unique up to the action of the
stabilizer $G_\Sigma$ of $\Sigma$ in $G$ on the set of roots of $\Sigma$.
\end{proposition}

\begin{proof}
By \cite[33.17]{TW}, (i) holds in case ${\sf F}$. 
By \ref{abc107}, it follows that (i) holds also in case ${\sf B}$
when $K=L$. Simply by restricting
scalars, we conclude that (i) holds in ${\sf B}$ also when $L\ne K$.

Suppose now that we are in case ${\sf G}$. 
Let $C$ and $C_1$ be the two chambers in $\Sigma$ and suppose
that $\alpha=\{C\}$. Let $\dot{\alpha}_i$ be the elements of 
$\dot{\Phi}$ ordered clockwise modulo 12 so that
$\dot{\alpha}_1,\ldots,\dot{\alpha}_6$ are the roots in 
the set $\dot{\Phi}_\alpha$ defined in \ref{abc49} (where $\dot{\Phi}$, $\dot{\alpha}_i$,
etc.~are as in \ref{abc48}).
Let $U_i$ denote the root group $U_{\dot{\alpha}_i}$ 
of $\dot{\Delta}$ for all $i$, let $(x_i)_{i\in[1,6]}$ denote the collection of isomorphisms
$x_i=x_{\dot{\alpha}_i}$ from the additive group of $K$
to $U_i$ which appear in the relations \ref{abc6a}--\ref{abc6c} and let 
$U_+$ denote the subgroup of ${\rm Aut}(\dot{\Delta})$
generated by the root groups $U_i$ for all $i\in[1,6]$.
By \ref{abc6}.i, $x_i(t)^\rho=x_{7-i}(t)$ for all $i\in[1,6]$ and 
by \ref{abc49}, $U_\alpha$ is the centralizer of $\rho$ in $U_+$.
Let 
\begin{equation}\label{abc55z}
x_\alpha(r,s,t)=x_1(r)x_2(r^{\theta+1}-s)x_3(t+rs)x_4(r^{\theta+2}-rs+t)x_5(-s)x_6(r)
\end{equation}
for all $(r,s,t)\in T$. Thus $x_\alpha$ is a map from $T$ to $U_+$.
By \ref{abc6}--\ref{abc6c}, \ref{abc72} and a bit of calculation,
this map is, in fact, an isomorphism from $T$ to $U_\alpha$.

Let $G^\dagger$ be the subgroup of ${\rm Aut}(\dot{\Delta})$ as defined
in \ref{abc13}.
Since the elements in the stabilizer $G^\dagger_{C,C_1}$ 
are type-preserving and fix opposite chambers of 
the apartment $\dot{\Sigma}$,
they are contained in the subgroup $H$ of ${\rm Aut}(\dot{\Delta})$ 
defined in \ref{abc81}. 
If $g\in G^\dagger_{C,C_1}$, then by \ref{abc81}, there exist $z,z_1,z_2\in K^*$ such that 
$x_\alpha(r,s,t)^g=x_\alpha(zr,z_1s,z_2t)$ for all $(r,s,t)\in T$.
Since $x_\alpha$ is an isomorphism,
the map $(r,s,t)\mapsto(zr,z_1s,z_2t)$ is an automorphism of $T$.
It follows that $z_1=z^{\theta+1}$ and $z_2=z^{\theta+2}t$. 
Thus for each element $g$ of $G^\dagger$ fixing the two chambers $C$ and $C_1$
of the apartment $\Sigma$, there exists an element $z\in K^*$ such that
\begin{equation}\label{abc55y}
x_\alpha(r,s,t)^g=x_\alpha(zr,z^{\theta+1}s,z^{\theta+2}t)
\end{equation}
for all $(r,s,t)\in T$. 

Let the group $T$ be identified with the set $\Delta\backslash\{C\}$ 
via the map sending $a\in T$ to the image of $C_1$ under the element $x_\alpha(a)$
of $U_\alpha$. Even though we are using
exponential notation (and, by implication, composition from left to right)
in the claim we are proving, for the remainder of this proof we will think of the group $G$
as acting on the set $C\cup T$ from the left (with composition
from right to left) in order to conform with the notation in \cite{tits-ree} 
from where we borrow the following argument. For the same reason, 
we will also use additive notation for $T$ (only in this proof)
even though it is not an abelian
group; in particular, we let $0$ denote the identity $(0,0,0)$ of $T$. Thus
$\Sigma=\{C,0\}$ and for each $a\in T$, the element $x_\alpha(a)$ of $U_\alpha$
fixes $C$ and induces the map $b\mapsto a+b$ on $T$. 

For each element $a=(r,s,t)\in T^*$, we set
\begin{equation}\label{abc55k}
u(a)=r^2s-rt+s^\theta-r^{\theta+3}
\end{equation}
and 
\begin{equation}\label{abc55m}
v(a)=r^\theta s^\theta-t^\theta+rs^2+st-r^{2\theta+3}.
\end{equation}
(These expressions are taken from \cite[5.3]{tits-ree}, where $N(a)$
is called $w=w(a)$. Note that in \cite[5.3]{tits-ree}, there is an
exponent $\theta$ missing in the second term of $w$ and a minus sign
missing in front of the whole formula for $w$; see Section~2.10 in
\cite{tom}.) As is explained in Section~5 of \cite{tits-ree}, $N(a)\ne0$ if $a\ne0$
and there is an element $\omega$ in $G$ interchanging the two chambers
$C$ and $0$ of $\Sigma$ such that 
\begin{equation}\label{abc55n}
\omega(a)=\big(-v(a)/N(a),-u(a)/N(a),-t/N(a)\big)
\end{equation}
for all $a=(r,s,t)\in T^*$.
By \ref{abc72}, \ref{abc55k}, \ref{abc55m} and \ref{abc55n} (and a bit of calculation), 
we have
\begin{equation}\label{abc55j}
N(\omega(a))=N(a)^{-1}
\end{equation}
and 
\begin{equation}\label{abc55p}
N(-a)=N(a)
\end{equation}
for all $a\in T^*$ (where $-a$ is the inverse of $a$ in $T$).

The element $\omega^2$ fixes $0$ and $C$. 
By \ref{abc55y}, there thus exists $z\in K^*$ such that
\begin{equation}\label{abc55x}
\omega^2(r,s,t)=(zr,z^{\theta+1}s,z^{\theta+2}t)
\end{equation}
for all $(r,s,t)\in T$. Let $v=z^{\theta+2}$. Then 
$$(0,0,v)=\omega^2(0,0,1)=\omega(1,0,-1)=(0,0,1)$$
by \ref{abc55n} and therefore $z=v^{2-\theta}=1$. We conclude that 
\begin{equation}\label{abc55o}
\omega^2=1
\end{equation}
by \ref{abc55x}.

Now let $a=(r,s,t)\in T^*$ and set $a'=\omega(-\omega(a))$. (This makes sense since
$\omega$ maps $T^*$ to itself.)
By \ref{abc55o}, the two products $\omega x_\alpha(a)\omega$ and
$x_\alpha(\omega(a))\omega x_\alpha(a')$ both map $0$ to $0$ and $C$ to 
$\omega(a)$. Thus 
\begin{equation}\label{abc55a}
\rho_a:=\big(x_\alpha(\omega(a))\omega x_\alpha(a')\big)^{-1}\omega x_\alpha(a)\omega
\end{equation}
fixes both $C$ and $0$. By \ref{abc55o} and \ref{abc55a}, we have
\begin{equation}\label{abc55h}
\rho_a(\omega(-a))=-a'=-\omega(-\omega(a)).
\end{equation}

Let $\xi=\rho_a\omega$. Thus $\xi$ is an element of $G$ 
interchanging $C$ and $0$. By \ref{abc55o} and \ref{abc55a}, we have
$$\xi=x_\alpha(-a')\cdot\omega x_\alpha(-\omega(a))\omega\cdot
x_\alpha(a)\in U_{\alpha}^*\cdot U_{-\alpha}^*\cdot U_\alpha^*.$$
By \ref{abc8a}.i, therefore,
$$\xi=m_\Sigma(\omega x_\alpha(-\omega(a))\omega)$$
and 
$$x_\alpha(a)=\lambda(\omega x_\alpha(-\omega(a))\omega).$$
Hence by \ref{abc8a}.iii, $\xi=m_\Sigma(x_\alpha(a))$. Therefore
\begin{equation}\label{abc55i}
\rho_a=m_\Sigma(x_\alpha(a))\omega.
\end{equation}

By \ref{abc55y}, there exists $z\in K^*$ such that 
\begin{equation}\label{abc55u}
\rho_a(w,u,v)=(zw,z^{\theta+1}u,z^{\theta+2}v)
\end{equation}
for all $(w,u,v)\in T$. 
By \ref{abc55j} and \ref{abc55p}, we have $N(-\omega(a))=N(a)^{-1}$.
By \ref{abc55n} and \ref{abc55u}, therefore,
the third coordinate of $-\omega(-\omega(a))$ is $t$, whereas the 
third coordinate of $\rho_a(\omega(-a))$ is $tz^{\theta+2}/N(a)$. 
By \ref{abc55h}, it follows that
$z^{\theta+2}=N(a)$ and hence 
\begin{equation}\label{abc55q}
z=N(a)^{2-\theta}
\end{equation}
if $t\ne0$. If $t=0$, we obtain the same conclusion by comparing the first or 
second coordinates
of both sides of the identity \ref{abc55h}; we leave these calculations
to the reader.
By \ref{abc55u} and \ref{abc55q}, finally, we have
$\rho_a=1$ if $a=(0,0,1)$. By \ref{abc55o} and \ref{abc55i}, it follows
that $\omega=m_\Sigma(0,0,1)$.
Thus (ii) holds by \ref{abc55i}, \ref{abc55u} and \ref{abc55q}. (Note that
in (ii), $x_\alpha(w,u,v)^{h(r,s,t)}$ is to be interpreted as 
$x_\alpha(w,u,v)$ conjugated first by 
$m_\Sigma(x_\alpha(0,0,1))$ and then by $m_\Sigma(x_\alpha(r,s,t))$, whereas
in the proof $\rho_a(w,u,v)$ is to be interpreted as the image of $(w,u,v)$ under
$m_\Sigma(x_\alpha(0,0,1))$ to which then the map $m_\Sigma(x_\alpha(r,s,t))$
is applied.)
\end{proof}

\section{Bruhat-Tits spaces for the Ree and Suzuki groups}

We begin this section with a result which explains why a valuation of the root
datum of a Suzuki-Ree building defined over a pair $(K,\theta)$ (or triple
$(K,\theta,L)$) requires the existence of a $\theta$-invariant valuation of $K$.
We then formulate our most important result in \ref{abc93}.

\begin{theorem}\label{abc90}
Let $\Delta$, $\Sigma$, $\Phi$, $\alpha$, $x_\alpha$, etc.~be as in \ref{abc48} and \ref{abc55},
let $w=x_\alpha(0,1)$ in cases ${\sf B}$ and ${\sf F}$ and let $w=x_\alpha(0,0,1)$
in case ${\sf G}$, 
let $\psi$ be a valuation of the root datum of $\Delta$ based at $\Sigma$
and let 
$$\varphi=\psi-\psi_\alpha(w)\alpha/(\alpha\cdot\alpha)$$ 
(as defined in \ref{abc30}).
(Thus $\varphi$ is a valuation equipollent to $\psi$ such that $\varphi_\alpha(w)=0$.)
Then there exists a unique $\theta$-invariant valuation $\nu$
of $K$, which depends only on the equipollence class of $\psi$, such that
\begin{equation}\label{abc90a}
\varphi_\alpha(x_\alpha(s,t))=\nu\big(R(s,t)\big)
\end{equation}
for all $(s,t)\in S^*$ in cases ${\sf B}$ and ${\sf F}$ and
\begin{equation}\label{abc90b}
\varphi_\alpha(x_\alpha(r,s,t))=\nu\big(N(r,s,t)\big)
\end{equation}
for all $(r,s,t)\in T^*$ in case ${\sf G}$.
\end{theorem}

\begin{proof}
Let 
$$\nu(t)=\varphi_\alpha(x_\alpha(0,t^\theta))/2$$
for all $t\in K^*$ in cases ${\sf B}$ and ${\sf F}$ and let
$$\nu(t)=\varphi_\alpha(x_\alpha(0,0,t))/2$$
for all $t\in K^*$ in case ${\sf G}$.
Then \ref{abc90a} and \ref{abc90b} hold 
by \ref{abc40} with $g=w$ and $u=x_\alpha(s,t)$ or $x_\alpha(r,s,t)$ and \ref{abc55}.
It thus need only to show that $\nu$ is a $\theta$-invariant valuation of $K$.

Let $\nu(0)=\infty$. 
By (V1), we have $\nu(s+t)\ge\min\{\nu(s),\nu(t)\}$ for all $s,t\in K$.
By \ref{abc55} again, we have
$$x_\alpha(0,s^\theta)^{m_\Sigma(w)m_\Sigma(x_\alpha(0,t^\theta))}
=x_\alpha(0,s^\theta t^{2\theta})$$ 
for all $s,t\in K^*$ in cases ${\sf B}$ and ${\sf F}$ and
$$x_\alpha(0,0,s)^{m_\Sigma(w)m_\Sigma(x_\alpha(0,0,t))}=x_\alpha(0,0,st^2)$$
for all $s,t\in K^*$ in case ${\sf G}$.
By \ref{abc40} again, this time with $g=x_\alpha(0,t^\theta)$ or $x_\alpha(0,0,t)$
and $u=x_\alpha(0,s^\theta)$ or $x_\alpha(0,0,s)$, it follows that
$$\nu(s^2t)=\nu(t)+2\nu(s)$$
for all $s,t\in K^*$ in all three cases.
>From this identity, we thus obtain
$$\nu(s^2t^2)=\nu(t^2)+2\nu(s)$$
for all $s,t\in K^*$ and (setting $t=1$)
$$\nu(s^2)=2\nu(s)$$
for all $s\in K^*$. Therefore 
$$\nu(st)=\nu(s)+\nu(t)$$
for all $s,t\in K^*$. 
Since $|\varphi_\alpha(U_\alpha^*)|>1$ by \ref{abc14}, it follows from
\ref{abc90a} and \ref{abc90b} that $|\nu(K^*)|>1$. 
Thus $\nu$ is a valuation of $K$.

It remains only to show that $\nu(u)\ge0$ implies that $\nu(u^\theta)\ge0$.
Let $u$ be an element of $K^*$ such that $\nu(u)\ge0$ (and hence $\nu(u^2)=2\nu(u)\ge0$).
Suppose first that we are in case ${\sf B}$ or ${\sf F}$, so 
$$(1,0)\cdot(0,u^\theta)=(1,u^\theta)$$
in $S$. By (V1), therefore,
$$\varphi_\alpha(x_\alpha(1,u^\theta))\ge\min\{\varphi_\alpha(x_\alpha(1,0)),
\varphi_\alpha(x_\alpha(0,u^\theta))\}.$$
Hence by \ref{abc90a},
$$\nu(1+u^\theta+u^2)\ge\min\{0,\nu(u^2)\}=0.$$
It follows that $\nu(u^\theta)\ge0$. 

Suppose now that we are in case ${\sf G}$, so 
$$(1,0,0)\cdot(0,0,u)=(1,0,u)$$
in $T$. By (V1), therefore,
$$\varphi_\alpha(x_\alpha(1,0,u))\ge\min\{\varphi_\alpha(x_\alpha(1,0,0)),
\varphi_\alpha(x_\alpha(0,0,u))\}.$$
Hence by \ref{abc90b},
$$\nu(u^2-u^\theta-1)\ge\min\{0,\nu(u^2)\}=0.$$
Again we conclude that $\nu(u^\theta)\ge0$.
\end{proof}

The converse of \ref{abc90} is also valid:

\begin{theorem}\label{abc91}
Let $\Delta$, $\Sigma$, $\Phi$, $\alpha$, $x_\alpha$, etc.~be as in 
\ref{abc48} and \ref{abc55}.
Suppose that $\nu$ is a $\theta$-invariant valuation of $K$ and let 
$\varphi_\alpha$ be the map given in \ref{abc90a} (in cases ${\sf B}$
and ${\sf F}$) or \ref{abc90b} (in case ${\sf G}$). Then $\varphi_\alpha$
extends to a valuation of the root datum of $\Delta$ based at $\Sigma$.
\end{theorem}

\medskip
It would not be hard to prove this result directly.
The principal difficulty is to show that
\begin{equation}\label{abc744}
\nu\big(R\big((s,t)\cdot(u,v)\big)\big)\ge\min\big\{\nu\big(R(s,t)\big),\nu\big(R(u,v)\big)\big\}
\end{equation}
for all $(s,t),(u,v)\in S$ and
\begin{equation}\label{abc744a}
\nu\big(N\big((r,s,t)\cdot(w,u,v)\big)\big)\ge
\min\big\{\nu\big(N(r,s,t)\big),\nu\big(N(w,u,v)\big)\big\}
\end{equation}
for all $(r,s,t),(w,u,v)\in T$. These inequalities are required to verify (V1). 

Rather than prove \ref{abc91} directly, however,
we will prove a stronger result (\ref{abc92}--\ref{abc93})
which will have \ref{abc91} (and thus also the two inequalities \ref{abc744}
and \ref{abc744a}) as corollaries. 
In Section~9 we include a direct proof of the inequalities
\ref{abc744} and \ref{abc744a} only because it might be of 
some independent interest. See also 9.1.10 in \cite{bruh-tits}.

\begin{notation}\label{abc92}
Let $\dot{\Phi}$, $\dot{\mathbb A}$, ${\mathbb A}$ and $\tau$ be as in \ref{abc400},
let $\dot{\Delta}$, $\Delta$, $\dot{\Sigma}$ and $\Sigma$ be as in \ref{abc48}, let 
$\dot{W}$ be the Weyl group of $\dot{\Phi}$,
let $W$ be the restriction 
of the centralizer of $\tau$ in $\dot{W}$ to the subspace ${\mathbb A}$ and
let ${\mathbb W}$ be the group of all isometries of ${\mathbb A}$ generated
by $W$ and all translations of ${\mathbb A}$. Let 
$\nu$ be a $\theta$-invariant valuation of $K$, let $\rho$ be the automorphism of 
$\dot{\Delta}$ described in \ref{abc6}.i, let $\dot{\Phi}$ be identified 
with the set of roots of $\dot{\Sigma}$ via the map called $\dot{\pi}$ in 
\ref{abc29} and let $\dot{\varphi}$ denote the
$\rho$-invariant valuation of the root datum of $\dot{\Delta}$ based at $\dot{\Sigma}$
determined by $\nu$ as described in \ref{abc80}. Let $(\dot{X},\dot{\mathcal A})$ 
be the Bruhat-Tits space of type $\dot{\Phi}$, $\dot{A}$ the apartment of 
$(\dot{X},\dot{\mathcal A})$ and $x_A$ the point of $\dot{A}$ obtained by applying
\ref{abc31} to $\dot{\Delta}$, $\dot{\Sigma}$ and $\dot{\varphi}$. The 
pair $(\dot{X},\dot{\mathcal A}_\tau)$ defined as in \ref{abc5} can also be thought of
as the Bruhat-Tits space of type $\dot{\Phi}$
obtained by applying \ref{abc31} to $\dot{\Delta}$, $\dot{\Sigma}$ and $\dot{\varphi}$,
but only after identifying $\dot{\Phi}$ with the set of roots of $\dot{\Sigma}$ via
$\dot{\pi}\circ\tau$ rather than $\dot{\pi}$.
By the uniqueness assertion in \ref{abc31}, there exists a $\tau$-automorphism
$\dot{\rho}$ of $(\dot{X},\dot{\mathcal A})$ (as defined in \ref{abc5})
that induces the automorphism $\rho$ on 
$\dot{\Delta}$ and maps the pair $(\dot{A},x_A)$ to itself. This map satisfies
\begin{equation}\label{abc92b}
\dot{\rho}\circ\dot{f}_{\dot{A},x_A}=\dot{f}_{\dot{A},x_A}\circ\tau,
\end{equation}
where $\dot{f}_{\dot{A},x_A}\in\dot{\mathcal A}$ is as defined in \ref{abc64}.
Let $\dot{\mathcal A}_\rho$ denote the set of charts $\dot{f}\in\dot{\mathcal A}$ such that
$\dot{\rho}$ maps 
the apartment $\dot{f}(\dot{\mathbb A})$ to itself and acts trivially on 
$\dot{f}({\mathbb A})$, let 
$${\mathcal A}=\{\dot{f}|_{\mathbb A}\mid\dot{f}\in\dot{\mathcal A}_\rho\},$$
and let
\begin{equation}\label{abc92a}
X=\bigcup_{\dot{f}\in\dot{\mathcal A}_\rho}\dot{f}({\mathbb A}).
\end{equation}
Thus $X$ is contained in the fixed point set $\dot{X}^{\dot{\rho}}$
of $\dot{\rho}$. 
By \ref{abc92b}, $\dot{f}_{\dot{A},x_A}\in\dot{\mathcal A}_\rho$ and hence
\begin{equation}\label{abc92c}
A:=\dot{f}_{\dot{A},x_A}({\mathbb A})
\end{equation}
is a subset of $X$ and $x_A=\dot{f}_{\dot{A},x_A}(0)$ is a point of $A$. 
\end{notation}

Here now is our main result:

\begin{theorem}\label{abc93}
Let $\Phi$, $\nu$, $(X,{\mathcal A})$, $A$ and $x_A$ be as in \ref{abc92}.
Thus, in particular, $\nu$ is a $\theta$-invariant valuation of $K$.
Then the following hold:
\begin{thmenum}
\item The pair $(X,{\mathcal A})$ is a Bruhat-Tits space of type $\Phi$ 
whose building at infinity
is $\Delta$ (in the sense of \ref{abc24} if the rank of $\Delta$ is one) and 
$A$ is an apartment of $(X,{\mathcal A})$.
\item Let $\alpha\in\Phi$ be as in \ref{abc55} 
and let $\varphi$ be the valuation of the root datum of $\Delta$
based at $\Sigma$ that appears in \ref{abc20}.ii
when \ref{abc20} (and then \ref{abc23})
is applied to the triple $(X,{\mathcal A})$, $A$ and $x_A$.
Then there exists a valuation $\nu_1$ equivalent to $\nu$ such that
$\varphi_\alpha$ satisfies \ref{abc90a} or \ref{abc90b} with $\nu_1$
in place of $\nu$.
\end{thmenum}
\end{theorem}

\noindent
Note that \ref{abc91} is a consequence of \ref{abc93}. 

\section{The Proof of \ref{abc93}}

For the rest of this paper, we let $\dot{\Delta}$, $\Delta$,
$(\dot{X},\dot{\mathcal A})$, $\dot{\Phi}$,
$(\dot{\mathbb A},\dot{\mathbb W})$, $\tau$, $\rho$, $\Phi$, $({\mathbb A},{\mathbb W})$, 
$(X,{\mathcal A})$, $A$, $x_A$,
$\dot{\rho}$, $\dot{f}_\rho$, etc.~be as in \ref{abc92}. Let $\dot{X}^{\dot{\rho}}$
denote the set of fixed points of $\dot{\rho}$. 
We prove \ref{abc93} in a series of steps.

\begin{proposition}\label{abc95}
Suppose that $\dot{A}_1$ is an apartment of 
$(\dot{X},\dot{\mathcal A})$. Then the following are equivalent:
\begin{thmenum}
\item $\dot{A}_1$ is $\dot{\rho}$-invariant
and contains at least one sector fixed by $\dot{\rho}$. 
\item There exists a chart $\dot{f}_1$ 
in $\dot{\mathcal A}_\rho$ such that $\dot{\rho}\circ\dot{f}_1=\dot{f}_1\circ\tau$.
\item $\dot{A}_1$ is the image of a chart in $\dot{\mathcal A}_\rho$. 
\end{thmenum}
If $\dot{f}_1$ is as in (ii), then 
$\dot{X}^{\dot{\rho}}\cap\dot{A}_1=X\cap\dot{A}_1=\dot{f}_1({\mathbb A})$.
\end{proposition}

\begin{proof}
Let $\dot{f}$ be a chart in $\dot{\mathcal A}$
such that $\dot{A}_1=\dot{f}(\dot{\mathbb A})$. Suppose first that $\dot{A}_1$
is $\dot{\rho}$-invariant and that there exists a sector $S$ of $\dot{\Phi}$ 
such that
\begin{equation}\label{abc95a}
\dot{\rho}(\dot{f}(S))=\dot{f}(S).
\end{equation}
By \ref{abc5}, there exists $\dot{f}'\in\dot{\mathcal A}$ such that
$\dot\rho\circ\dot{f}=\dot{f}'\circ\tau$.
Since $\dot{A}_1$ is $\dot{\rho}$-invariant, also $\dot{f}'(\dot{\mathbb A})$ 
equals $\dot{A}_1$.
By (A2) (in \ref{abc1}), therefore, there exists $\dot{w}\in\dot{\mathbb W}$ such that
$\dot{f}'=\dot{f}\circ\dot{w}$. Thus
\begin{equation}\label{abc95b}
\dot{\rho}\circ\dot{f}=\dot{f}\circ\dot{w}\circ\tau.
\end{equation}
By \ref{abc95a}, it follows that $\dot{w}\circ\tau$ fixes $S$.
Thus $\dot{w}\circ\tau$
is a non-trivial automorphism of $\Sigma(\dot{\Phi})$ fixing $S$. 
There is only one such automorphism. Since $\tau$ also fixes a sector 
of $\dot{\Phi}$ and $\dot{W}$ acts transitively on the set of sectors of
$\dot{\Phi}$, it follows that there exists $\dot{w}_1\in\dot{W}$ such that
$\dot{w}\circ\tau=\dot{w}_1\circ\tau\circ\dot{w}_1^{-1}$. Let $\dot{f}_1=\dot{f}\circ\dot{w}_1$.
Then $\dot{f}_1\in\dot{\mathcal A}$ by (A1) and
\begin{align*}
\dot{\rho}\circ\dot{f}_1&=\dot{\rho}\circ\dot{f}\circ\dot{w}_1\\
&=\dot{f}\circ\dot{w}\circ\tau\circ\dot{w}_1\\
&=\dot{f}\circ\dot{w}_1\circ\tau=\dot{f}_1\circ\tau
\end{align*}
by \ref{abc95b}. 
It follows from this identity that
$\dot{\rho}$ acts trivially on $\dot{f}_1({\mathbb A})$,
so $\dot{f}_1\in\dot{\mathcal A}_\rho$ (but $\dot{\rho}$ 
does not fix any other points in $\dot{A}_1$, so 
$\dot{X}^{\dot{\rho}}\cap\dot{A}_1=X\cap\dot{A}_1=\dot{f}_1({\mathbb A})$). Thus
(i) implies (ii). It now suffices to observe that
if $\dot{f}_1$ is a chart
in $\dot{\mathcal A}_\rho$ whose image is $\dot{A}_1$,
then by \ref{abc98}.iv, $\dot{\rho}$ fixes $\dot{f}_1(S)$
for every $\tau$-invariant sector $S$ of $\dot{\Phi}$.
\end{proof}

\begin{proposition}\label{abc95x}
Let $\dot{f}$ and $\dot{f}_1$ be two charts in $\dot{\mathcal A}_\rho$ such
that $\dot{A}_1:=\dot{f}(\dot{\mathbb A})=\dot{f}_1(\dot{\mathbb A})$. Then 
$$\dot{f}({\mathbb A})=\dot{f}_1({\mathbb A})=\dot{X}^{\dot{\rho}}\cap\dot{A}_1=X\cap\dot{A}_1.$$
\end{proposition}

\begin{proof}
By \ref{abc95}, it suffices to assume that $\dot{\rho}\circ\dot{f}_1=\dot{f}_1\circ\tau$
and $\dot{X}^{\dot{\rho}}\cap\dot{A}_1=X\cap\dot{A}_1=\dot{f}_1({\mathbb A})$. Thus
$\dot{f}({\mathbb A})\subset\dot{X}^{\dot{\rho}}\cap\dot{A}_1=\dot{f}_1({\mathbb A})$.
By (A2), therefore, there exists $\dot{w}\in\dot{\mathbb W}$ such that
$\dot{f}_1\circ\dot{w}$ and $\dot{f}$ coincide on ${\mathbb A}$. Thus 
$\dot{f}_1(\dot{w}({\mathbb A}))=\dot{f}({\mathbb A})
\subset\dot{f}_1({\mathbb A})$
and hence $\dot{w}$ maps ${\mathbb A}$ to itself. Therefore
$\dot{f}_1({\mathbb A})=\dot{f}({\mathbb A})$.
\end{proof}

\begin{proposition}\label{abc96}
Let $x\in X$ and let $g_x$ be as in \ref{abc59}. 
Then the set of fixed points of $\dot{\rho}$ in the
residue $(\dot{X},\dot{\mathcal A})_x$ has the structure of 
a building of type $\Phi$ and for all $\dot{f}\in\dot{\mathcal A}_\rho$
mapping $0$ to $x$, $g_x(\dot{f}({\mathbb A}))$ is an apartment 
of this building.
\end{proposition}

\begin{proof} 
By \ref{abc59}, $(\dot{X},\dot{\mathcal A})_x$ is a building
of type $\dot{\Phi}$ whose apartments are all of the form
$g_x(\dot{f}(\dot{\mathbb A}))$ for some $\dot{f}\in\dot{\mathcal A}$. 
The conclusion holds, therefore,
by applying \ref{abc29} to this building and the automorphism
of this building induced by $\dot{\rho}$.
\end{proof}

\begin{corollary}\label{abc28}
Let $x\in X$, let $\dot{u}$ be a germ at $x$ fixed by
$\dot{\rho}$ and let $\dot{f}$ be a chart in $\dot{\mathcal A}_\rho$
mapping $0$ to $x$. Then there exists a sector with vertex $x$
in the apartment $g_x(\dot{f}(\dot{\mathbb A}))$ that is 
fixed by $\dot{\rho}$ and
whose germ is opposite a maximal germ at $x$ containing $\dot{u}$.
\end{corollary}

\begin{proof}
The germ $\dot{u}$ is a face
and $g_x(\dot{f}({\mathcal A}))$ is an apartment of the building 
of type $\Phi$ described in \ref{abc96}. 
Since the apartment $g_x(\dot{f}(\dot{\mathbb A}))$ is $\dot{\rho}$-invariant, 
a sector with vertex $x$ in this apartment is fixed by $\dot{\rho}$ if and only 
if its germ is fixed by $\dot{\rho}$.
The claim follows, therefore, from the fact that for each 
chamber $C$ and each apartment $\Sigma$ in a spherical building,
there always exists a chamber in $\Sigma$ which is opposite $C$.
\end{proof}

\begin{proposition}\label{abc177}
Let $\dot{A}_1$ be the image of a chart $\dot{f}_1$ in $\dot{\mathcal A}_\rho$.
Then the map 
$$\dot{S}_1\mapsto\dot{S}_1\cap X$$ 
is a bijection from the set of 
sectors of $\dot{A}_1$ that are fixed by $\dot{\rho}$ to the set of 
sectors of $\dot{f}_1({\mathbb A})$, i.e.~to the set of images under $\dot{f}_1$
of sectors of $\Phi$.
\end{proposition}

\begin{proof}
By \ref{abc98}, the map $\dot{S}\mapsto\dot{S}\cap{\mathbb A}$ is a bijection 
from the set of sectors of $\dot{\Phi}$ that are fixed by $\tau$ to the 
set of sectors of $\Phi$. If $\dot{S}$ is any one of these sectors, then
$\dot{f}_1(\dot{S}\cap{\mathbb A})=\dot{f}_1(\dot{S})\cap X$ by \ref{abc95x}.
\end{proof}

\begin{proposition}\label{abc76}
The pair $(X,{\mathcal A})$ is a non-discrete Euclidean building of type $\Phi$.
\end{proposition}

\begin{proof} 
We need to show that $(X,{\mathcal A})$ satisfies the conditions (A1)--(A6) formulated in
\ref{abc1}.

Let $f\in{\mathcal A}$ and $w\in{\mathbb W}$. There exist elements
$\dot{f}$ in $\dot{\mathcal A}_\rho$ and $\dot{w}$ in the centralizer
of $\tau$ in $\dot{\mathbb W}$ such that 
$f$ is the restriction of $\dot{f}$ to ${\mathbb A}$ and 
$w$ is the restriction of $\dot{w}$ to ${\mathbb A}$.
Since $(\dot{X},\dot{\mathcal A})$ satisfies (A1), 
we have $\dot{f}\circ\dot{w}\in\dot{\mathcal A}$.
It follows that $f\circ w\in\dot{\mathcal A}_\rho$ since 
$\dot{f}(\dot{w}(\dot{\mathbb A}))=\dot{f}(\dot{\mathbb A})$ and 
$\dot{f}(\dot{w}({\mathbb A}))=\dot{f}({\mathbb A})$.
Thus $(X,{\mathcal A})$ satisfies (A1).

Next let $\dot{f},\dot{f}'\in\dot{\mathcal A}_\rho$.
Since $(\dot{X},\dot{\mathcal A})$ satisfies (A2), the set
$$\dot{M}:=\{v\in\dot{\mathbb A}\mid \dot{f}(v)\in\dot{f}'(\dot{\mathbb A})\}$$
is closed and convex and there exists $\dot{w}\in\dot{\mathbb W}$ such that the maps
$\dot{f}$ and $\dot{f}'\circ\dot{w}$ coincide on $\dot{M}$. Let
$$M:=\{v\in{\mathbb A}\mid\dot{f}(v)\in\dot{f}'({\mathbb A})\}.$$
Then $M\subset\dot{M}\cap{\mathbb A}$. Let $v\in\dot{M}\cap{\mathbb A}$. 
Thus $v\in{\mathbb A}$ and $\dot{f}(v)=\dot{f}'(v')$ for some $v'\in\dot{\mathbb A}$.
Since $\dot{f}\in\dot{\mathcal A}_\rho$, the point $\dot{f}(v)$ is a point
of $\dot{f}'(\dot{\mathbb A})$ fixed by 
$\dot{\rho}$. Since also $\dot{f}'\in\dot{\mathcal A}_\rho$, it follows 
by \ref{abc95x} that $v'\in{\mathbb A}$ and hence $v\in M$.
We conclude that $M=\dot{M}\cap{\mathbb A}$. Since $\dot{M}$ is closed and convex,
it follows that $M$ is also closed and convex.

To finish showing that $(X,{\mathcal A})$ satisfies (A2), we can assume that
$|M|>1$. By (A1), we can assume further that the origin $0$ is contained
in $M$ and that $\dot{w}$ fixes $0$, so $\dot{w}\in\dot{W}$. Let 
$\dot{f}_*$ and $\dot{f}'_*$ be as in \ref{abc63} 
and let $\Xi$ denote the building of type $\Phi$ described in \ref{abc96}. 
By \ref{abc98}.iv, we can think of the $\tau$-invariant faces of $\Sigma(\dot{\Phi})$ 
as the faces of $\Sigma(\Phi)$. Thus $\dot{f}_*$ and $\dot{f}'_*$ both
map $\Sigma(\Phi)$ to apartments of the building $\Xi$. 
In cases ${\sf B}$ and ${\sf G}$,
an apartment of $\Xi$ is an arbitrary two-element set of chambers.
In case ${\sf F}$, an apartment of $\Xi$ 
is a circuit consisting of 16 chambers and 16 panels and two distinct apartments
intersect either in a connected piece of this circuit (possibly empty)
or in two opposite panels. 
Let $Y$ be the set of all faces of $\Sigma(\Phi)$ which are mapped by $\dot{f}_*$ to 
$\dot{f}'_*(\Sigma(\Phi))$. 
Thus either $Y=\Sigma(\Phi)$, $Y$ is a simplicial arc in $\Sigma(\Phi)$ (possibly
empty) or we are in case ${\sf F}$ and $Y$ consists of two opposite panels (i.e.~two opposite
non-maximal faces).
The map $(\dot{f}'_*)^{-1}\circ\dot{f}_*$ from $Y$ into $\Sigma(\Phi)$ is 
type-preserving (since the maps $\dot{f}_*$ and $\dot{f}'_*$ are type-preserving) and if 
$Y$ consists of two opposite panels of a given type, then 
$(\dot{f}'_*)^{-1}\big(\dot{f}_*(Y)\big)$ consists of two opposite panels
of the same type (since the maps $\dot{f}_*$ and $\dot{f}'_*$ both map opposite panels
to opposite panels).
Thus in every case there exists an element $\dot{w}_1$
in the centralizer of $\tau$ in $\dot{W}$ such that $\dot{f}_*$ 
and $\dot{f}_*'\circ \dot{w}_1$ coincide on $Y$.

Let $z$ be an arbitrary non-zero element of $M$ and let $z'=\dot{w}(z)$. Since $M$ is convex,
it contains the closed interval $[0,z]$. Let $\dot{F}$ be the unique minimal
face of $\Sigma(\Phi)$ that contains $[0,z]$. Since $\dot{f}$ and $\dot{f}'\circ\dot{w}$
coincide on $M$, we have $z'\in{\mathbb A}$ and $\dot{f}(z)=\dot{f}'(z')$. If
$\dot{F}'$ is the minimal face of $\Sigma(\Phi)$
that contains $[0,z']$, then $\dot{f}(\dot{F})=\dot{f}'(\dot{F}')$.
Therefore $\dot{F}\in Y$ (and, in particular, $Y$ is not empty). Thus
$\dot{f}_*'(\dot{F}')=\dot{f}_*(\dot{F})=\dot{f}'_*(\dot{w}_1(\dot{F}))$ by the conclusion of the
previous paragraph and hence $\dot{F}'=\dot{w}_1(\dot{F})$. It follows that
$\dot{w}(z)=z'\in\dot{w}_1(\dot{F})$, so the points $z$ and 
$\dot{w}_1^{-1}\dot{w}(z)$ are both contained in the face $\dot{F}$. Since every point
of $\dot{\mathbb A}$ distinct from the origin 
is contained in at most one face of $\dot{\Phi}$ of a given type and the 
stabilizer of a face in $\dot{W}$ acts trivially on that face, it 
follows that $\dot{w}_1(z)=\dot{w}(z)$. 
We conclude that $\dot{f}'\circ \dot{w}_1$ coincides with $\dot{f}$ on $M$. 
Thus $(X,{\mathcal A})$ satisfies (A2).

Now let $x,x'\in X$. Since $(\dot{X},\dot{\mathcal A})$ satisfies (A3), there exists
an apartment $\dot{A}_1$ of $(\dot{X},\dot{\mathcal A})$
containing the interval $[x,x']$ (as defined in \ref{abc73}). 
By \ref{abc92a}, there exist $\dot{f},\dot{f}'\in\dot{\mathcal A}_\rho$
such that $x\in\dot{f}({\mathbb A})$ and $x'\in\dot{f}'({\mathbb A})$.
Since $\dot{\rho}$ fixes $x$ and $x'$, it fixes the point $g_x(x')$ of the residue
$(\dot{X},\dot{\mathcal A})_x$ defined in \ref{abc59}. Let $\dot{F}$ be the 
unique minimal face of the apartment $\dot{A}_1$ with vertex $x$
whose germ contains the point $g_x(x')$
and let $\dot{u}$ denote the germ of $\dot{F}$.
The germ $\dot{u}$ is fixed by $\dot{\rho}$ (since otherwise $\dot{u}^{\dot{\rho}}$ 
would be disjoint from $\dot{u}$). By \ref{abc28}, the apartment
$g_x(\dot{f}(\dot{\mathbb A}))$ contains a sector with vertex $x$
which is both fixed by $\dot{\rho}$ and whose germ is opposite a maximal germ 
at $x$ containing $\dot{u}$. By
Lemma~1.13 in \cite{parreau}, there exists an apartment $\dot{A}_2$ containing
$\dot{S}_1\cup\dot{F}$. Thus $[x,x']\subset\dot{F}\subset\dot{A}_2$.
The convex closure of $\dot{S}_1\cup\{x'\}$ is a sector of $\dot{A}_2$ 
with vertex $x'$ which contains the interval $[x,x']$.
We denote this sector by $\dot{S}_2$.
Since $\dot{\rho}$ fixes $\dot{S}_1$ and $x'$, it fixes $\dot{S}_2$
as well. By a second application of \ref{abc28}, there exists a
sector $\dot{S}_3$ with vertex $x'$ 
in the apartment $\dot{f}'(\dot{\mathbb A})$ that is fixed by
$\dot{\rho}$ and whose germ is opposite the germ of $\dot{S}_2$
at $x'$. By Lemma~1.12 in \cite{parreau}, there exists
a unique apartment $\dot{A}_3$ containing $\dot{S}_2\cup\dot{S}_3$.
This apartment contains $[x,x']$ (since $\dot{S}_2$ contains $[x,x']$)
and is $\dot{\rho}$-invariant (since $\dot{S}_2$ and 
$\dot{S}_3$ are $\dot{\rho}$-invariant). By \ref{abc95}, 
there exists a chart $\dot{f}_1$ in $\dot{\mathcal A}_\rho$
such that $\dot{A}_3=\dot{f}_1(\dot{\mathbb A})$ and $x,x'\in\dot{f}_1({\mathbb A})$.
Thus $(X,{\mathcal A})$ satisfies (A3).

Let $S$ and $S'$ be two sectors of $X$. By \ref{abc177}, there exist
$\dot{\rho}$-invariant sectors $\dot{S}$ and $\dot{S}'$ of
$(\dot{X},\dot{\mathcal A})$ such that $S=\dot{S}\cap X$ and
$S'=\dot{S}'\cap X$. The chambers $\dot{S}^\infty$ and $(\dot{S}')^\infty$
of $(\dot{X},\dot{\mathcal A})^\infty$ are contained in a 
$\rho$-invariant apartment of $(\dot{X},\dot{\mathcal A})^\infty$. This apartment is of the 
form $\dot{A}_1^\infty$, where $\dot{A}_1$ is a $\dot{\rho}$-invariant
apartment of $(\dot{X},\dot{\mathcal A})$. The apartment
$\dot{A}_1$ contains
sectors $\dot{S}_1$ and $\dot{S}_1'$ such that $\dot{S}_1\subset\dot{S}$ and 
$\dot{S}_1'\subset\dot{S}'$. Let $\dot{S}_2=\dot{S}_1\cap\dot{\rho}(\dot{S}_1)$ and 
$\dot{S}_2'=\dot{S}_1'\cap\dot{\rho}(\dot{S}_1')$. The intersection of two
subsectors of a given sector is again a subsector. It follows that
$\dot{S}_2$ is a subsector of $\dot{S}$ and $\dot{S}_2'$ is a subsector of $\dot{S}'$. 
Since $\dot{\rho}^2=1$, both of these subsectors are fixed by $\dot{\rho}$. 
By \ref{abc95}, therefore,
$\dot{A}_1=\dot{f}_1(\dot{\mathbb A})$ for some 
$\dot{f}_1\in\dot{\mathcal A}_\rho$. Hence by \ref{abc177},
$\dot{S}_2\cap X$ and $\dot{S}_2'\cap X$ are subsectors of $S$ and $S'$ 
both are contained in $\dot{f}_1({\mathbb A})$. Thus $(X,{\mathcal A})$ satisfies (A4).

We turn now to (A5). A root of $(X,{\mathcal A})$ is the image
under a chart in $\dot{\mathcal A}_\rho$ of $H_\alpha$ for some $\alpha\in\Phi$,
where $H_\alpha$ is as in \ref{abc60}. Let $\dot{A}_1$ and 
$\dot{A}_2$ be the images of two charts in $\dot{\mathcal A}_\rho$ such that
$\dot{A}_1\cap\dot{A}_2\cap X$ contains a root $\beta$ of $(X,{\mathcal A})$
but $\dot{A}_1\cap X\ne\dot{A}_2\cap X$.
By (A2), the set $\dot{A}_1\cap\dot{A}_2\cap X$ is closed. We can therefore
choose a point $x$ in this set such that the apartments $\dot{\Sigma}_1:=g_x(\dot{A}_1)$
and $\dot{\Sigma}_2:=g_x(\dot{A}_2)$ of the residue of $(\dot{X},\dot{\mathcal A})_x$
are distinct. 
For $i=1$ and $2$, 
let $\dot{\Sigma}_i^{\dot{\rho}}$ be the set of chambers of 
$\dot{\Sigma}_i$ fixed by $\dot{\rho}$. Then
$\dot{\Sigma}_1^{\dot{\rho}}$ and
$\dot{\Sigma}_2^{\dot{\rho}}$ both span apartments of the building $\Xi$
of type $\Phi$ defined in \ref{abc96}. Let $S$
be a sector contained in the root $\beta$. Then the convex hull $S'$ 
of $\{x\}\cup S$ is a sector with vertex $x$ and
$\dot{\Sigma}_1^{\dot{\rho}}$ and
$\dot{\Sigma}_2^{\dot{\rho}}$ both contain the unique chamber of $\Xi$
that contains $g_x(S')$.
It follows that $\dot{\Sigma}_1^{\dot{\rho}}\cap\dot{\Sigma}_2^{\dot{\rho}}$
contains a root of both $\dot{\Sigma}_1^{\dot{\rho}}$
and $\dot{\Sigma}_2^{\dot{\rho}}$. On the other hand, 
$\dot{\Sigma}_1^{\dot{\rho}}\ne\dot{\Sigma}_2^{\dot{\rho}}$ since
otherwise $\dot{\Sigma}_1\cap\dot{\Sigma}_2$ would contain a pair
of opposite chambers. Two distinct apartments of a spherical building
whose intersection contains a root intersect in a root and their
symmetric difference spans a third apartment.
Thus the symmetric
difference of $\dot{\Sigma}_1^{\dot{\rho}}$ and $\dot{\Sigma}_2^{\dot{\rho}}$
spans a third apartment of $\Xi$. Let $\dot{u}_1$ and $\dot{u}_2$ be two
chambers in this apartment that are opposite in $\Xi$ and hence
also opposite in the residue $(X,{\mathcal A})_x$. There exist
unique sectors $\dot{S}_1$ and $\dot{S}_2$ of 
$\dot{A}_1$ and $\dot{A}_2$ with vertex $x$ 
whose germs are $\dot{u}_1$ and $\dot{u}_2$,
and by Proposition~1.12 in \cite{parreau} there exists an apartment
$\dot{A}_3$ containing these two sectors. 
In particular, $x\in\dot{A}_1\cap\dot{A}_2\cap\dot{A}_3\cap X$.
Now suppose that $\dot{A}'$ is the image of a chart $\dot{f}'$ in $\dot{\mathcal A}_\rho$ 
such that both $\dot{A}'\cap\dot{A}_1\cap X$ and $\dot{A}'\cap\dot{A}_2\cap X$ contain
roots of $(X,{\mathcal A})$ that are disjoint from $\dot{A}_1\cap\dot{A}_2\cap X$.
By \ref{abc177}, $\dot{f}'({\mathbb A})$ contains subsectors
of both $\dot{S}_1\cap X$ and $\dot{S}_2\cap X$ and hence
$\dot{A}'$ contains subsectors of both
$\dot{S}_1$ and $\dot{S}_2$. Thus $\dot{A}'=\dot{A}_3$
since $\dot{A}_3$ is the convex hull of 
any two sectors, one contained in $S_1$ and the other in $S_2$. 
Therefore $x\in\dot{A}_1\cap\dot{A}_2\cap\dot{A}'$. Thus $(X,{\mathcal A})$
satisfies (A5).

Let $\dot{d}$ be the metric on $\dot{X}$ that appears in (A6).
The restriction of $\dot{d}$ to $X$ is a metric on $X$, and 
each chart in ${\mathcal A}$ 
is the restriction to ${\mathbb A}$ of a chart in $\dot{\mathcal A}$.
Therefore $(X,{\mathcal A})$ satisfies (A6) with the restriction of $\dot{d}$
to $X$ in place of $d$. 
\end{proof}

\begin{proposition}\label{abc100}
The building $\Delta$ is the building at infinity of $(X,{\mathcal A})$ (in 
the sense of \ref{abc24} in cases ${\sf B}$ and ${\sf G}$). 
\end{proposition}

\begin{proof}
By \ref{abc177}, there is a canonical bijection $\pi$
from the chamber set of $(X,{\mathcal A})^\infty$
to the set of all chambers $\dot{S}^\infty$ of $(\dot{X},\dot{\mathcal A})^\infty=\dot{\Delta}$
such that $\dot{S}$ is a sector of $(\dot{X},\dot{\mathcal A})$
that is fixed by $\dot{\rho}$. If 
$\dot{S}_1$ is an arbitrary sector of $(\dot{X},\dot{\mathcal A})$
such that $\dot{S}_1^\infty$ is fixed by $\rho$, then 
$\dot{S}_2:=\dot{S}_1\cap\dot{S}_1^{\dot{\rho}}$ is a sector of 
$(\dot{X},\dot{\mathcal A})$ fixed by $\dot{\rho}$ such that
$\dot{S}_2^\infty=\dot{S}_1^\infty$. It follows that $\pi$
is, in fact, an isomorphism from $(X,{\mathcal A})^\infty$ to $\Delta$.
Furthermore, the conclusions
of \ref{abc20} hold for $(X,{\mathcal A})$ and $\Delta$, even
in cases ${\sf B}$ and ${\sf G}$, by 
\ref{abc20} applied to $(\dot{X},\dot{\mathcal A})$ and \ref{abc49}.
\end{proof}

By \ref{abc76} and \ref{abc100}, we conclude that \ref{abc93}.i holds. 
Now let $\alpha\in\Phi$ and $\varphi$ be as in \ref{abc93}.ii.

\begin{proposition}\label{abc101}
There exists a positive real number $k$ (which depends on the case) such that
$$\varphi_\alpha(x_\alpha(0,t))=k\cdot\nu(t)$$ 
for all $t\in K^*$ (or $t\in L^*$) in case ${\sf B}$ or ${\sf F}$ and 
$$\varphi_\alpha(x_\alpha(0,0,t))=k\cdot\nu(t)$$ 
for all $t\in K^*$ in case ${\sf G}$.
\end{proposition}

\begin{proof}
Suppose first that we are in case ${\sf G}$. We use the notation from the 
proof of \ref{abc55}. By \ref{abc55z}, in particular, we have
$$x_\alpha(0,0,t)=x_{\dot{\alpha}_4}(t)x_{\dot{\alpha}_5}(t)$$ 
for all $t\in K$. Since $\tau$ interchanges $\dot{\alpha}_4$ and $\dot{\alpha}_5$,
there exists a positive real number $k$ such that for all $t\in K^*$, the
affine half-spaces 
$H_{\dot{\alpha}_4,\nu(t)}$ and $H_{\dot{\alpha}_5,\nu(t)}$ 
(as defined in \ref{abc20}) both
intersect ${\mathbb A}$ (which is the space of fixed points of $\tau$)
in an affine half-space of the form $H_{\alpha,k\cdot\nu(t)}$. 
Choose $t\in K^*$, let $u=x_{\dot{\alpha}_4}(t)$,
let $y=x_{\dot{\alpha}_5}(t)$ and let $z=uy$. Then 
$$\dot{A}\cap\dot{A}^u=\dot{f}_{\dot{A},x_A}(H_{\dot{\alpha}_4,\nu(t)})$$
and 
$$\dot{A}\cap\dot{A}^y=\dot{f}_{\dot{A},x_A}(H_{\dot{\alpha}_5,\nu(t)})$$
by \ref{abc20}.ii and \ref{abc80}. Suppose that $z$ fixes a point $x$
in ${\mathbb A}$ which is not in 
$$\dot{f}_{\dot{A},x_A}(H_{\alpha,k\cdot\nu(t)}).$$
Choose $x'$ in 
$$\dot{f}_{\dot{A},x_A}\big(H_{\dot{\alpha}_4,\nu(t)}\cap H_{\dot{\alpha}_5,\nu(t)}\big)$$
but not in $\dot{f}_{\dot{A},x_A}({\mathbb A})$. Since $u$ and $y$ both fix $x'$, so does $z$.
Thus $z$ fixes every point in the interval $[x,x']$. This interval contains
points, however, which are in 
$\dot{f}_{\dot{A},x_A}(H_{\dot{\alpha}_4,\nu(t)})$ or 
$\dot{f}_{\dot{A},x_A}(H_{\dot{\alpha}_5,\nu(t)})$ but not in both.
These points are fixed by $u$ or $y$ but not both. Hence they cannot be 
fixed by $z$. We conclude that 
$$A\cap A^z=\dot{f}_{\dot{A},x_A}(H_{\alpha,k\cdot\nu(t)}).$$
By \ref{abc64}, $f_{A,x_A}$ is the restriction of $\dot{f}_{\dot{A},x_A}$
to ${\mathbb A}$. Thus $\varphi_\alpha(x_\alpha(0,0,t))=k\cdot\nu(t)$ for all $t\in K^*$.

The proof in cases ${\sf B}$ and ${\sf F}$ is virtually the same as the proof in case
${\sf G}$; we leave the details to the reader. 
\end{proof}

By \ref{abc90}, there exists
a $\theta$-invariant valuation $\nu_1$ such that $\varphi_\alpha$
satisfies \ref{abc90a} or \ref{abc90b} with $\nu_1$ in place of $\nu$.
By \ref{abc101}, it follows that $\nu_1$ is equivalent to $\nu$.
Thus \ref{abc93}.ii holds. This concludes the proof of \ref{abc93}.

\bigskip
\centerline{*\qquad\qquad*\qquad\qquad*}
\medskip
Here, finally, is a precise version of the remark Tits made in \cite{tits-como} 
that was discussed in the introduction.

\begin{theorem}\label{abc99}
Bruhat-Tits spaces whose building at infinity (in the sense of \ref{abc24} if
the rank of $\Delta$ is one) 
is a given Suzuki-Ree building $\Delta$ (in the sense of \ref{abc97})
defined over a triple $(K,\theta,L)$
(in case ${\sf B}$) or pair $(K,\theta)$ (in cases ${\sf F}$ or ${\sf G}$)
are classified by $\theta$-invariant valuations of $K$. 
Equivalent $\theta$-invariant valuations of $K$ correspond 
to equivalent Bruhat-Tits spaces (in the sense of \ref{abc67}).
\end{theorem}

\begin{proof}
Let $\Delta$ be a Suzuki-Ree building defined over the triple $(K,\theta,L)$
in case ${\sf B}$ or the pair $(K,\theta)$ in cases ${\sf F}$ or ${\sf G}$,
and let $\alpha$ and $x_\alpha$ be as in \ref{abc55} (with respect to
some apartment $\Sigma$ of $\Delta$).
Suppose that $\nu$ is a $\theta$-invariant valuation of $K$. By
\ref{abc91}, there exists a valuation $\varphi$ of the root datum of $\Delta$
based at $\Sigma$ satisfying \ref{abc90a} or \ref{abc90b}.
By \ref{abc84}, $\varphi$ is unique up to equipollence. By \ref{abc31}, $\varphi$
determines a unique Bruhat-Tits space $(X,{\mathcal A})$ having $\Delta$ as
its building at infinity. By \ref{abc23}, any valuation 
equipollent to $\varphi$ determines the same Bruhat-Tits space. 

Suppose, conversely, that $(X,{\mathcal A})$
is a Bruhat-Tits space whose building at infinity is $\Delta$.
By \ref{abc31}, $(X,{\mathcal A})$ determines a valuation 
of the root datum of $\Delta$ based at $\Sigma$ which is unique
up to equipollence. By \ref{abc90}, this equipollence class determines 
a unique $\theta$-invariant valuation $\nu$ of $K$ such that
\ref{abc90a} or \ref{abc90b} holds for some valuation in this 
equipollence class. 
\end{proof}

\section{Appendix}
The inequalities \ref{abc744} and \ref{abc744a} hold by \ref{abc91} and 
the condition (V1). 
In this section we give an elementary proof of these inequalities
which might be of independent interest.
In fact, we only give a proof of \ref{abc744a}; the interested reader
will have no trouble applying the same strategy to the inequality \ref{abc744}.
The proof we give is based on a suggestion of Theo Grundh\"ofer.

We suppose that we are in case ${\sf G}$ and that $\nu$
is a $\theta$-invariant valuation of $K$. As in \ref{abc72}, we have
\begin{equation}\label{abc744c}
N(r,s,t)=r^{\theta+1}s^\theta-rt^\theta-r^{\theta+3}s-r^2s^2+s^{\theta+1}+t^2-r^{2\theta+4}
\end{equation}
for all $(r,s,t)$ in the group $T$. 

\begin{lemma}\label{abc746}
Let $(r,s,t)\in T$ and suppose that the minimum of $\nu(r)$, $\nu(s)$ and $\nu(t)$
is $0$. Then $\nu\big(N(r,s,t)\big)=0$.
\end{lemma}

\begin{proof}
The Tits endomorphism $\theta$ induces a Tits endomorphism of $\bar K$. We let
$\bar\theta$ denote this endomorphism and let $\bar N$
be the map obtained by applying the formula \ref{abc744c} to
the pair $(\bar K,\bar\theta)$ rather than $(K,\theta)$.  
By \ref{abc750}, $\bar N$ is anisotropic.
\end{proof}

\begin{lemma}\label{abc745}
Let $(r,s,t)\in T$, let $A=(2\sqrt{3}+4)\nu(r)$, let $B=(\sqrt{3}+1)\nu(s)$, let 
$C=2\nu(t)$ and let $M$ be the minimum of $A$, $B$ and $C$. Then 
$\nu\big(N(r,s,t)\big)=M$.
\end{lemma}

\begin{proof}
We can assume that $(r,s,t)\ne(0,0,0)$. Suppose first that $M=A$. Then 
$$N(1,s/r^{\theta+1},t/r^{\theta+2})=N(r,s,t)/r^{2\theta+4}$$
by \ref{abc744c}. Moreover, $\nu(s/r^{\theta+1})$ and $\nu(t/r^{\theta+2})$ are both
non-negative since $B\ge A$ and $C\ge A$. Hence 
$$\nu\big(N(1,s/r^{\theta+1},t/r^{\theta+2})\big)=0$$
by \ref{abc746}. It then follows that
$\nu\big(N(r,s,t)\big)=\nu(r^{2\theta+4})=A=M$. 

Suppose next that $M=C$. In this case, we observe that
$$N(r/t^{2-\theta},s/t^{\theta-1},1)=N(r,s,t)/t^2.$$
Moreover, $\nu(r/t^{2-\theta})$ and $\nu(s/t^{\theta-1})$ are both
non-negative. Hence 
$$\nu\big(N(r/t^{2-\theta},s/t^{\theta-1},1)\big)=0$$
by \ref{abc746}. It then follows that $\nu\big(N(r,s,t)\big)=\nu(t^2)=C=M$.

It suffices now to assume that $M=B$ and that $B$ is strictly
less than both $A$ and $C$. In this case, $\nu(s^{\theta+1})$
is strictly less than the value under
$\nu$ of each of the remaining six terms on the right hand side of \ref{abc744c}.
Therefore $\nu\big(N(r,s,t)\big)=B=M$.
\end{proof}

Now let $(r,s,t),(w,u,v)\in T$ and let $M$ be the smaller of the two 
constants obtained by applying \ref{abc745} first to $(r,s,t)$
and then to $(w,u,v)$. Thus 
\begin{equation}\label{abc778a}
\nu(r),\nu(w)\ge M/(2\sqrt{3}+4)
\end{equation}
as well as 
\begin{equation}\label{abc778b}
\nu(s),\nu(u)\ge M/(\sqrt{3}+1)
\end{equation}
and 
\begin{equation}\label{abc778c}
\nu(t),\nu(v)\ge M/2.
\end{equation}
As in \ref{abc72}, we have
$$(r,s,t)\cdot(w,u,v)=(r+w,s+u+r^\theta w,t+v-ru+sw-r^{\theta+1}w).$$
Let $a=(2\sqrt{3}+4)\nu(r+w)$, let $b=(\sqrt{3}+1)\nu(s+u+r^\theta w)$
and let 
$$c=2\nu(t+v-ru+sw-r^{\theta+1}w).$$ 
By \ref{abc778a}--\ref{abc778c}, we have $(2\sqrt{3}+4)\nu(x)\ge M$ for $x=r$ and $x=w$;
$(\sqrt{3}+1)\nu(x)\ge M$ for $x=s$, $x=u$ and $x=r^\theta w$; and 
$2\nu(x)\ge M$ for $x$ equal to each of the five terms in the sum
$$t+v-ru+sw-r^{\theta+1}w.$$ Hence $a,b,c\ge M$. By \ref{abc745}, therefore,
$$\nu\big((r,s,t)\cdot(w,u,v)\big)\ge M
={\rm min}\big\{\nu\big(N(r,s,t)\big),\nu\big(N(w,u,v)\big)\big\}.$$

\end{document}